\documentclass[12pt,oneside,english]{amsart}
\usepackage[T1]{fontenc}
\usepackage[latin1]{inputenc}
\usepackage{geometry}
\usepackage{textcomp}
\usepackage{amsthm}

\numberwithin{equation}{section} 
\numberwithin{figure}{section} 
\theoremstyle{plain}
  \theoremstyle{plain}
  \newtheorem*{conjecture*}{Conjecture}
\theoremstyle{plain}
\newtheorem{thm}{Theorem}
  \theoremstyle{plain}
  \newtheorem{cor}[thm]{Corollary}
  \theoremstyle{remark}
  \newtheorem{rem}[thm]{Remark}
  \theoremstyle{definition}
  \newtheorem{defn}[thm]{Definition}
  \theoremstyle{plain}
  \newtheorem{prop}[thm]{Proposition}
  \theoremstyle{plain}
  \newtheorem{lem}[thm]{Lemma}

\usepackage{amsthm}

\usepackage{amssymb}

\usepackage{amssymb}

\usepackage{amssymb}

\def\makebbb#1{
    \expandafter\gdef\csname#1\endcsname{
        \ensuremath{\Bbb{#1}}}
}
\makebbb{R}
\makebbb{N}
\makebbb{Z}
\makebbb{C}
\makebbb{G}
\makebbb{E}
\makebbb{H}
\makebbb{P}
\makebbb{B}
\makebbb{K}
\makebbb{E}

\usepackage{babel}

\usepackage{babel}

\usepackage{babel}

\usepackage{babel}

\usepackage{babel}

\begin{document}

\title{Sharp inequalities for determinants of Toeplitz operators and $\bar{\partial}-$Laplacians
on line bundles }

\author{Robert J. Berman}

\curraddr{Mathematical Sciences, Chalmers University of Technology and the
University of Gothenburg, SE-412 96 Göteborg, Sweden}

\email{robertb@chalmers.se}
\begin{abstract}
We prove sharp inequalities for determinants of Toeplitz operators
and twisted $\bar{\partial}-$Laplace operators on the two-sphere,
generalizing the Moser-Trudinger-Onofri inequality. In particular
a sharp version of conjectures of Gillet-Soulé and Fang motivated
by Arakelov geometry is obtained; applications to $SU(2)-$invariant
determinantal random point processes on the two-sphere are also discussed.
The inequalities are obtained as corollaries of a general theorem
about the maximizers of a certain non-local functional defined on
the space of all positively curved Hermitian metrics on an ample line
bundle $L$ over a compact complex manifold. This functional is an
{}``adjoint version'', introduced by Berndtsson, of Donaldson's
$L$-functional and generalizes the Ding-Tian functional whose critical
points are Kähler-Einstein metrics. In particular, new proofs of some
results in Kähler geometry are also obtained, including a lower bound
on Mabuchi's $K-$energy and the uniqueness result for Kähler-Einstein
metrics on Fano manifolds of Bando-Mabuchi. 
\end{abstract}
\maketitle
\tableofcontents{}

\section{Introduction}

Consider the two-dimensional sphere $S^{2}$ equipped with its standard
Riemannian metric $g_{0}$ of constant positive curvature, normalized
so that the corresponding volume form $\omega_{0}$ gives unit volume
to $S^{2}.$ A celebrated inequality of Moser-Trudinger-Onofri proved
in its sharp form by Onofri \cite{on}, asserts that \begin{equation}
\log\int_{S^{2}}e^{-u}\omega_{0}\leq-\int_{S^{2}}u\omega_{0}+\frac{1}{4}\int_{S^{2}}du\wedge d^{c}u\label{eq:m-t-o intro}\end{equation}
 for any, say smooth, function $u$ on $S^{2},$ where the last term
is the $L^{2}-$norm of the gradient of $u$ in the conformally invariant
notation of section \ref{sub:Setup} below. 

As is well-known the inequality above has a rich geometric content
and appears in a number of seemingly unrelated contexts ranging from
the problem of prescribing the Gauss curvature in a conformal class
of metrics on $S^{2}$ (the Yamabe and Nirenberg problems \cite{cha})
to sharp critical \emph{Sobolev inequalites} \cite{bec} and lower
bounds on \emph{free energy functionals} in mathematical physics \cite{on,k}.
The geometric content of the inequality above appears clearly when
considering the extremal funtions $u.$ Note first that $e^{-u}\omega_{0}$
appearing in the left hand side above is the volume form corresponding
to the metric $g_{u}:=e^{-u}g_{0},$ conformally equivalent to $g_{0}.$
Denoting by $\mbox{Conf}{}_{0}(g_{0})$ the set of all metrics $g_{u}$
with normalized volume (equal to one), \emph{equality} holds in \ref{eq:m-t-o intro}
for $u$ such that $g_{u}\in\mbox{Conf}{}_{0}(g_{0})$ precisely when
$g_{u}$ is the pull-back of $g_{0}$ under a conformal transformation
of $S^{2}.$ Since, $g_{0}$ has constant curvature, this latter fact
means that $u$ satisfies the constant positive curvature equation
\[
\omega_{0}+dd^{c}u=e^{-u}\omega_{0},\]
where $dd^{c}u$ is proportional to $(\Delta_{g_{0}}u)\omega_{0}$
(using the notation in section \ref{sub:Setup}).

There is also a \emph{spectral} intepretation of the Moser-Trudinger-Onofri
inequality. As shown by Onofri \cite{on} and Osgood-Phillips-Sarnak
\cite{o-p-s} the inequality \ref{eq:m-t-o intro} is equivalent to
the fact that the functional \[
g_{u}\mapsto\det\Delta_{g_{u}}\]
 on $\mbox{Conf}{}_{0}(g_{0}),$ where $\det\Delta_{g_{u}}$ denotes
the (zeta function regularized) \emph{determinant of the Laplacian}
$\Delta_{g_{u}}$ wrt the metric $g_{u},$ achieves its upper bound
precisely for $g_{0}$ (modulo conformal transformations as above).
The bridge between this latter fact and the inequality \ref{eq:m-t-o intro}
is given by the \emph{Polyakov anomaly formula} \cite{cha}, which
first appeared in Physics in the path integral (random surface) approach
to the quantization of the bosonic string.

From the point of view of complex geometry $(S^{2},g_{0})$ may be
identified with the complex projective line $\P^{1}$ endowed with
its standard $SU(2)-$invariant Kähler metric $\omega_{0}$ (the Fubini-Study
metric). The two-form $\omega_{0}$ is the normalized curvature form
of an Hermitian metric $h_{0}$ on the hyper plane line bundle $\mathcal{O}(1)\rightarrow\P^{1}.$
For any natural number $m,$ the pair $(\omega_{0},h_{0})$ induces
naturally
\begin{itemize}
\item an $SU(2)-$invariant Hermitian product on the space $H^{0}(\P^{1},\mathcal{O}(m))$
of global holomorphic sections of the $m$ th tensor power $\mathcal{O}(m)$
of $\mathcal{O}(1),$ i.e. on the space of all homogenous polynomials
of degree $m.$ 
\item a \emph{Dolbeault Laplace operator} $\Delta_{\bar{\partial}_{0}}^{(m)}$
(i.e. a $\bar{\partial}-$Laplacian) on the space of smooth sections
of $\mathcal{O}(m),$ such that its null-space is precisely $H^{0}(\P^{1},\mathcal{O}(m)).$ 
\end{itemize}
Changing the Hermitian metric on $\mathcal{O}(m)$ corresponds to
{}``twisting'' by a function $e^{-u},$ for $u\in\mathcal{C}^{\infty}(S^{2})$
and we will denote the corresponding Dolbeault Laplace operator by
$\Delta_{\bar{\partial}_{u}}^{(m)}$ (see section \ref{sub:Application-to-determinants}
for precise definitions). It is worth emphasizing that as opposed
to the Laplacian $\Delta_{g_{u}}$ the Dolbeault Laplacian $\Delta_{\bar{\partial}_{u}}^{(m)}$
is \emph{invariant} under translation of $u.$ Motivated by Arakelov
Geometry - notably the arithmetic Riemann-Roch theorem - Gillet-Soulé
made a general conjecture which in the case of $S^{2}$ amounts to
the following (\cite{g-s}; see also \cite{g-sa} p. 526-527) 
\begin{conjecture*}
(Gillet-Soulé). The determinant of the Dolbeault Laplacian $\Delta_{\bar{\partial}_{u}}^{(m)}$
naturally induced by the function $u$ on any given line bundle $\mathcal{O}(m)$
over $S^{2}$ is bounded from above when $u$ ranges over $\mathcal{C}^{\infty}(S^{2}).$ 
\end{conjecture*}
This was confirmed by Fang \cite{f}, who by symmetrization reduced
the problem to the case when $u$ is invariant under rotation around
an axes of $S^{2},$ earlier treated by Gilllet-Soulé \cite{g-s}.
Fang also put forward the following more precise form of the conjecture
above:
\begin{conjecture*}
(Fang). The upper bound in the previous conjecture is achieved precisely
for $u$ identically constant.
\end{conjecture*}
As pointed out by Fang one motivation for this latter conjecture is
that, after introducing suitable numerical constants depending on
$m$ in the right hand side of \ref{eq:m-t-o intro}, it is implied
be an inequality whose formulation is obtained by replacing $\int_{X}e^{-u}\omega_{0}$
by the determinant of the \emph{Toeplitz operator with symbol $e^{-u}$}
acting on the space $H^{0}(\P^{1},\mathcal{O}(m)).$ In this latter
form the conjecture can be seen as a holomorphic analogue of an inequality
appearing in connection to the classical \emph{Szegö strong limit
theorem} on $S^{1}$ (see chapter 3.1 in \cite{g-sz}). The relation
between inequalities of Toeplitz operators on the sphere and upper
bounds on determinants of Dolbeault Laplacians is a direct consequence
of the anomaly formula of Bismut-Gillet-Soulé \cite{b-g-s}, which
generalizes Polyakov's formula referred to above. It should also be
pointed out that Toeplitz operators appear naturally in the \emph{Berezin-Toeplitz
quantization} of Kähler manifolds and in \emph{microlocal analysis
}\cite{a-e}. 

In this paper the positive solution of Fang's conjecture will be deduced
from a general result about the maximizers of a non-local functional
$\mathcal{F}_{\omega_{0}}$ defined on the space of all positively
curved Hermitian metrics on an ample line bundle $L$ over a Kähler
manifold $(X,\omega_{0}).$ In fact, a more precise inequality then
the one conjectured by Fang will be obtained (Corollary \ref{cor:moser})
which implies both Fang's conjecture \emph{and} the Moser-Trudinger-Onofri
inequality above (and hence the extremal properties of $\det\Delta_{g_{u}},$
as well). The inequality obtained is equivalent to the upper bound
\begin{equation}
\log(\frac{\det\Delta_{\bar{\partial}_{u}}^{(m)}}{\det\Delta_{\bar{\partial}_{0}}^{(m)}})\leq-\frac{1}{2}(\frac{1}{(m+2)})\int du\wedge d^{c}u(\leq0),\label{eq:sharp ineq for det in intr}\end{equation}
 where $\det\Delta_{\bar{\partial}_{u}}^{(m)}$ is the Dolbeault Laplacian
corresponding to $\mathcal{O}(m),$ which clearly implies Fang's conjecture
above. The extremals in the first inequality above will also be characterized. 

As pointed out in remark \ref{rem:sharp} the inequality in Corollary
\ref{cor:moser} is \emph{sharp} in a rather strong sense. Moreover,
in the limit when $m$ tends to infinity, while the function $u$
is kept fixed, the inequality becomes an asymptotic \emph{equality.}
As it turns out, this latter fact is essentially equivalent to a Central
Limit Theorem for a certain random point process on the sphere. This
process appears naturally as a \emph{random matrix model} and as a
\emph{one component plasma }in the statistical physics litterature
(see section \ref{sec:Application-to-determinantal}).

The functional $\mathcal{F}_{\omega_{0}}$ referred to above is an
{}``adjoint version'', introduced by Berndtsson, of Donaldson's
(normalized) $L$-functional and generalizes the Ding-Tian functional
whose critical points are Kähler-Einstien metrics. In particular,
new proofs of some results in Kähler geometry are also obtained, including
a lower bound on Mabuchi's $K-$energy and the uniqueness result for
Kähler-Einstein metrics on Fano manifolds of Bando-Mabuchi (see section
\ref{sub:Further-relations-to} for precise references). 

The relation between the inequality \ref{eq:m-t-o intro} and Kähler-Einstein
metrics of positive curvature in higher dimensions seems to first
have been suggested by Aubin \cite{au}. It should also be pointed
out that recently Rubinstein \cite{rub,rub0} gave a different complex
geometric proof of the inequality \ref{eq:m-t-o intro} using the
inverse Ricci operator and its relation to various energy functionals
in Kähler geometry. See also Müller-Wendland \cite{m} for a proof
of the result on extremals of determinants of the scalar Laplacian
using the Ricci flow. However, these latter methods seem to be less
well adapted to the non-local variational equations which appear in
the setting of Gilllet-Soulé's and Fang's conjectures.

Before turning to the precise statement of the main result we will
first introduce the general setup.

\subsection{Setup\label{sub:Setup}}

Let $L\rightarrow X$ be a holomorphic line bundle over a compact
complex manifold $X$ of complex dimension $n.$ Denote by $\mbox{Aut}_{0}(X,L)$
the group of automorphism of $(X,L)$ in the connected component of
the identity, modulo the elements that act as the identity on $X.$
The line bundle $L$ will be assumed \emph{ample,} i.e. there exists
a Kähler form $\omega_{0}$ in the first Chern class $c_{1}(L)$ and
a {}``weight'' $\psi_{0}$ on $L$ such that $\omega_{0}$ is the
normalized curvature $(1,1)-$form of the hermitian metric on $L$
locally represented as $h_{0}=e^{-\psi_{0}}.$ In this notation, the
space of all positively curved smooth hermitian metrics on $L$ may
be identified with the open convex subset \[
\mathcal{H}_{\omega_{0}}:=\{u:\,\omega_{u}:=dd^{c}u+\omega_{0}>0\}\]
 of \emph{$\mathcal{C}^{\infty}(X),$} where $d^{c}:=i(-\partial+\overline{\partial})/4\pi,$
so that $dd^{c}=\frac{i}{2\pi}\partial\overline{\partial}.$ Note
that, under this identification, the natural action of $\mbox{Aut}_{0}(X,L)$
on the space of all metrics on $L$ corresponds to the action $(u,F)\mapsto v:=F^{*}(\psi_{0}+u)-\psi_{0}$
so that, in particular, $\omega_{v}=F^{*}\omega_{u}.$ Occasionally,
we will also work with the closure $\overline{\mathcal{H}}_{\omega_{0}}$
of $\mathcal{H}_{\omega_{0}}$ in $L^{1}(X,\omega_{0}),$ coinciding
with the space of all $\omega_{0}-$plurisubharmonic functions on
$X,$ i.e. the space of all upper semi-continuous functions $u$ which
are absolutely integrable and such that $\omega_{u}\geq0$ as a $(1,1)-$current. 

We equip the $N-$dimensional complex vector space $H^{0}(X,L+K_{X})$
of all holomorphic sections of the adjoint bundle $L+K_{X}$ where
$K_{X}$ is the canonical line bundle on $X,$ with the Hermitian
product induced by $\psi_{0},$ i.e. \[
\left\langle s,s\right\rangle _{\psi_{0}}:=i^{n^{2}}\int_{X}s\wedge\bar{s}e^{-\psi_{0}},\]
identifying $s$ with a holomorphic $n-$form with values in $L.$
We will use additive notation for tensor products of line bundles.

\subsection{Statement of the main results}

Next, we will introduce the two functionals on $\mathcal{H}_{\omega_{0}}$
which will play a leading role in the following. First, consider the
following energy functional \begin{equation}
\mathcal{E}_{\omega_{0}}(u):=\frac{1}{(n+1)!V}\sum_{i=1}^{n}\int_{X}u(dd^{c}u+\omega_{0})^{j}\wedge(\omega_{0})^{n-j},\label{eq:def of e intro}\end{equation}
where $V:=\mbox{Vol}(\omega_{0})$ is the volume of $L,$ which seems
to first have appeared in the work of Mabuchi \cite{m2} and Aubin
\cite{au}in Kähler geometry ($\mathcal{E}_{\omega_{0}}=-F_{\omega_{0}}^{0}$
in the notation of \cite{ti}). It also appears in Arithmetic (Arakelov)
geometry as the top degree component of the secondary Bott-Chern class
of $L$ attached to the Chern character. 

The second functional $\mathcal{L}_{\omega_{0}}$ may be geometrically
defined as $\frac{1}{N}$ times the logarithm of the quotient of the
volumes of the unit-balls in $H^{0}(X,L+K_{X})$ defined by the Hermitian
products induced by the metrics $\psi_{0}$ and $\psi_{0}+u$ \cite{b-b}.
Concretely, this means that \begin{equation}
\mathcal{L}_{\omega_{0}}(u):=-\frac{1}{N}\log\mbox{det}(\left\langle s_{i},s_{j}\right\rangle _{\psi_{0}+u}),\label{eq:def of l intro}\end{equation}
where $1\leq i,j\leq N$ and $s_{i}$ is any given base in $H^{0}(X,L+K_{X})$
which is orthogonal wrt $\left\langle \cdot_{i},\cdot\right\rangle _{\psi_{0}}.$
The functional $\mathcal{L}_{\omega_{0}}(u)$ may also be invariantly
expressed as a \emph{Toeplitz determinant:}

\begin{equation}
\mathcal{L}_{\omega_{0}}(u):=-\frac{1}{N}\log\mbox{det}(T[e^{-u}]),\label{eq:l as toeplitz det intro}\end{equation}
 where $T[e^{-u}]$ is the \emph{Toeplitz operator with symbol $e^{-u}$}
defined as the linear operator $\Pi_{L}\circ e^{-u}\cdot$ on $H^{0}(X,L+K_{X}),$
expressed in terms of the the orthogonal projection $\Pi_{L}:\,\mathcal{C}^{\infty}(X)\rightarrow H^{0}(X,L+K_{X})$
(compare formula \ref{eq:toeplitz as k} in the appendix).  If $N=0$
we let $\mathcal{L}_{\omega_{0}}(u):=-\infty.$ The normalizations
are made so that the functional

\[
\mathcal{F}_{\omega_{0}}:=\mathcal{E}_{\omega_{0}}-\mathcal{L}_{\omega_{0}}\]
is invariant under addition of constants and hence descends to a functional
on the space of all Kähler metrics in $c_{1}(L).$ An element $u$
in $\mathcal{H}_{\omega_{0}}$ will be said to be \emph{critical}
(wrt $L+K_{X})$ if it is a critical point of the functional $\mathcal{F}_{\omega_{0}}$
on $\mathcal{H}_{\omega_{0}},$ i.e. if $u$ is a smooth solution
in $\mathcal{H}_{\omega_{0}}$ of the Euler-Lagrange equations $(d\mathcal{F}_{\omega_{0}})_{u}=0.$
These equations may be written as the highly non-linear Monge-Ampère
equation: \begin{equation}
\frac{1}{Vn!}(dd^{c}u+\omega_{0})^{n}=\beta(u),\label{eq:e-l for f}\end{equation}
 where $\beta(u)$ is the Bergman measure associated to $u$ (formula
\ref{eq:bergman meas} below). This latter measure depends on $u$
in a \emph{non-local} manner and is strictly positive precisely when
$L+K_{X}$ is \emph{globally generated,} i.e. when there, given any
point $x$ in $X,$ exists an element $s$ in $H^{0}(X,L+K_{X})$
such that $s(x)\neq0.$ For example, since $L$ is ample, this condition
holds when $L$ is replaced by $kL$ for $k$ sufficiently large. 

By definition, a critical point $u$ is a priori only a \emph{local}
extremum of $\mathcal{\mathcal{F}}_{\omega_{0}}.$ But the next theorem
relates \emph{global} maximizers of $\mathcal{\mathcal{F}}_{\omega_{0}}$
and its critical points:
\begin{thm}
\emph{\label{thm:main}Let $L$ be an ample line bundle such that
the adjoint line bundle $L+K_{X}$ is globally generated. Then the
absolute maximum of the functional} \textup{$\mathcal{\mathcal{F}}_{\omega_{0}}$
on $\mathcal{H}_{\omega_{0}}$ is attained at any critical point $u.$
Moreover, any smooth maximizer of $\mathcal{\mathcal{F}}_{\omega_{0}}$
on} $\overline{\mathcal{H}}_{\omega_{0}}$ \textup{is unique (up to
addition of constants) modulo the action of} $\mbox{Aut}_{0}(X,L).$
\emph{In particular, such a maximizer is critical.} 
\end{thm}
In the case when the ample line bundle $L=-K_{X},$ so that $X$ is
a Fano manifold, the space $H^{0}(X,L+K_{X})$ is one-dimensional
and hence $\mathcal{L}_{\omega_{0}}(u)=-\frac{1}{N}\log\int e^{-(u+\psi_{0})}.$
Then it is well-known that any critical point may be identified with
a Kähler-Einstein metric on $X$. 

It should be emphasized that the \emph{existence} of critical points
of $\mathcal{\mathcal{F}}_{\omega_{0}}$ is a very difficult issue
closely related to conjectures of Yau, Tian, Donaldson and others
in Kähler geometry \cite{ti,do1,th}. Even in the case $L=-K_{X}$
there are well-known examples already on complex surfaces, where critical
points do not exist.

Next, assume that $(X,L)$ is \emph{$K-$homogenous,} i.e. that $X$
admits a transitive action by a compact semi-simple Lie group $K,$
whose action on $X$ lifts to $L.$ We will then take $\omega_{0}$
as the unique Kähler form in $c_{1}(L)$ which is invariant under
the action of $K$ on $X.$ 
\begin{cor}
\label{cor:homeg}Let $L\rightarrow X$ be a $K-$homogenous ample
holomorphic line bundle over a compact complex manifold $X$ and denote
by $\omega_{0}$ be the unique $K-$invariant Kähler metric in $c_{1}(L).$
Then, for any function $u$ in $\mathcal{H}_{\omega_{0}}$ \[
-\mathcal{L}_{\omega_{0}}(u)\leq-\mathcal{E}_{\omega_{0}}(u)\]
 with equality iff the function $u$ is constant, modulo the action
of $\mbox{Aut}_{0}(X,L).$ 
\end{cor}
Surprisingly, specializing to the case when $X$ is a complex curve
(i.e. $n=1)$ allows one to take $u$ as \emph{any} smooth function
(which is \emph{not }true in higher dimensions, as shown in \cite{rub0}
in the case when $L=-K_{X};$ see remark \ref{rem:j-f}). More generally,
we can then take $u$ to be in the Sobolev space $W^{2,1}(X)$ of
all functions $u$ on $X$ such that $u$ and its differential $du$
are square integrable. 

We will next consider the homegenous case, i.e. when $X=\P^{1},$
the complex projective line (i.e. topologically $X=S^{2},$ the two-sphere)
and hence $K=SU(2).$ In this case any ample line bundle $L$ may
be written as $\mathcal{O}(k),$ where $k$ is a positive integer
and $H^{0}(\P^{1},\mathcal{O}(k)+K_{\P^{1}})=H^{0}(\P^{1},\mathcal{O}(k-2))$
may be identified with the space of all polynomials of at most degree
$m:=k-2$ on the affine piece $\C$ in $\P^{1}$ (assuming $k\geq2).$
Moreover, if we take $\psi_{0,k}(z)=k\log(1+z\bar{z})$ as the fixed
invariant weight on $\mathcal{O}(k),$ in the usual trivialization
over the affine piece $\C,$ then, under the identification above,
the Hermitian product on $H^{0}(\P^{1},\mathcal{O}(k)+K_{\P^{1}})$
may be written as \begin{equation}
\left\langle p_{m},p_{m}\right\rangle _{\psi_{0}+u}:=\int_{\C}\frac{\left|p_{m}\right|^{2}}{(1+z\bar{z})^{m}}e^{-u}\omega_{0}\label{eq:scalar pr on s2}\end{equation}
for $p_{m}(z)$ a polynomial on $\mbox{\C}$ of degree at most $m$
(compare section \ref{sub:Explicit-expression}). Hence, \[
\mathcal{L}_{\omega_{0,k}}(u):=(m+1)\mathcal{L}_{m}(u):=-\log\mbox{det}(c_{ij}\int_{\C}\frac{z^{i}\bar{z}^{j}}{(1+z\bar{z})^{m}}e^{-u}\omega_{0}),\]
 where $i,j=0,...,m$ and $1/c_{ij}=(m+1)\binom{m}{i}\binom{m}{j}.$ 
\begin{cor}
\label{cor:moser}Let $u$ be a function in the Sobolev space $W^{2,1}(S^{2})$
on the two-sphere $S^{2}$ and denote by $\omega_{0}$ the volume
form corresponding to the metric on $S^{2}$ with constant curvature
and volume one. Then \[
-\mathcal{L}_{m}(u)\leq-(m+1)\int_{S^{2}}u\omega_{0}+(\frac{m+1}{m+2})\frac{1}{2}\int_{S^{2}}du\wedge du^{c}\]
with equality iff there exists a Möbius transformation $M$ of $S^{2}$
such that $\omega_{u}=M^{*}\omega_{0}.$
\end{cor}
The case when $m=0,$ so that $-\mathcal{L}_{m}(u)=\log\int_{\C}e^{-u}\omega_{0},$
is precisely the celebrated Moser-Trudinger-Onofri inequality \ref{eq:m-t-o intro}.
The reduction of the proof of Corollary \ref{cor:moser} to Corollary
\ref{cor:homeg}, is based on properties of the projection operator
$P_{\omega}$ (formula \ref{eq:proj oper intro}).

\subsubsection{\label{sub:Application-to-determinants}Application to determinants
of $\bar{\partial}-$Laplace operators and Analytic Torsion}

Consider again the case when $X=\P^{1}$ is the complex projective
line equipped with the standard Kähler form $\omega_{0}.$ Any function
$u$ corresponds to a metric on $\mathcal{O}(m)$ with weight $m\psi_{0}+u,$
where $m$ is a fixed non-negative integer. Hence, the pair $(\omega_{0,}u)$
induces natural Hilbert norms on the space $\Omega^{0,q}(\mathcal{O}(m))$
of smooth $(0,q)-$forms with values in $\mathcal{O}(m),$ where $q=0,1.$
Denote by $\Delta_{\bar{\partial}_{u}}^{(m)}$ the corresponding $\bar{\partial}-$Laplace
(Dolbeault) operator acting on the space $\Omega^{0}(\mathcal{O}(m),$
i.e. $\Delta_{\bar{\partial}_{u}}^{(m)}=\bar{\partial}^{^{*}}\bar{\partial},$
where $\bar{\partial}^{^{*}}$ is the formal adjoint of the $\bar{\partial}-$operator
\[
\bar{\partial}:\,\,\Omega^{0,0}(\mathcal{O}(m))\rightarrow\Omega^{0,1}(\mathcal{O}(m))\]
Note that $\bar{\partial}^{^{*}}$ may be expressed in terms of the
adjoint $\bar{\partial}^{^{*},0}$ induced by $u=0$ as \[
\bar{\partial}^{^{*}}=e^{u}\bar{\partial}^{^{*},0}e^{-u}\]
 The zeta function regularized\emph{ determinant} of the operator
obtained by restricting $\Delta_{\bar{\partial}_{u}}^{(m)}$ to the
orthogonal complement of its kernel will be denoted by $\mbox{\ensuremath{\det}}\Delta_{\bar{\partial}_{u}}^{(m)}$
(compare \cite{b-g-s}). Given the result in the previous corollary,
the anomaly formula (i.e. a family Riemann-Roch-Grothendieck theorem)
of Bismut-Gillet-Soulé \cite{b-g-s} now implies the following positive
solution of Fang's conjecture
\begin{cor}
\label{cor:max of det}Given the line bundle $\mathcal{O}(m)\rightarrow\P^{1},$
the corresponding functional \[
u\mapsto\mbox{\ensuremath{\det\Delta_{\bar{\partial}_{u}}^{(m)}}}\]
on the space of all smooth functions $u$ on $\P^{1}$ attains its
maximum precisely for $u$ a constant function. 
\end{cor}
In fact, the proof of the previous Corollary, will give the stronger
statement that the inequality \ref{eq:sharp ineq for det in intr}
for $\mbox{\ensuremath{\det}}\Delta_{\bar{\partial}_{u}}$ stated
in the introduction holds and that this latter inequality is equivalent
to Corollary \ref{cor:moser}. Note that a direct consequence of the
previous corollary is the following reponse to a variant of Kac's
classical question {}``Can one hear the shape of a drum?\textquotedbl{}
\cite{ka}: if the $\bar{\partial}-$Laplacian on \emph{some} power
$\mathcal{O}(m)$ induced by a smooth metric $h$ on $\mathcal{O}(1)\rightarrow\P^{1}$
has the same spectrum (including multiplicities) as the $\bar{\partial}-$Laplacian
induced by the standard $SU(2)$ invariant metric $h_{0},$ then $h=Ch_{0}$
for a positive number $C.$ 

Finally it should be pointed out that in the general case of an ample
line bundle $L$ Theorem \ref{thm:main} yields a bound on the twisted
\emph{Ray-Singer analytic torsion} (see for example \cite{b-g-s})
associated to a semi-positively curved metric on $L$ in terms of
the corresponding \emph{Quillen metric} and the functional $\mathcal{E}.$
This is a direct consquence of the fact that $L+K_{X}$ is ample,
so that the higher cohomology groups $H^{q}(X,L+K_{X}),$ $q\geq1,$
vanish, combined with the anomaly formula of Bismut-Gillet-Soulé \cite{b-g-s}.
For the sake of brevity the details are omitted.

\subsection{\label{sub:Further-relations-to}Further relations to previous results}

In this case when $L=-K_{X}$ the first statement of Theorem \ref{thm:main}
is a result of Ding-Tian\cite{d-t} and the {}``uniqueness'' of
critical points (i.e. Kähler-Einstein metrics in this case) was proved
earlier by Bando-Mabuchi \cite{b-m}. See \cite{bbgz} for a generalization
of this latter result to funtions of {}``finite energy'', in the
case when $\mbox{Aut}_{0}(X,L)$ is discrete (compare remark \ref{rem:finite energy}). 

The extremal property of the critical points in Theorem \ref{thm:main}
can also be seen as an analog of a result of Donaldson (Theorem 2
in \cite{don1}) who furthermore assumed that $\mbox{Aut}_{0}(X,L)$
is discrete. In this latter setting the role of the space $H^{0}(X,L+K_{X})$
is played by $H^{0}(X,L)$ equipped with the scalar products induced
by the weight $\psi_{0}+u$ and the integration measure $(\omega_{u})^{n}/n!$
Note however that in Donaldson's setting the functional corresponding
to $\mathcal{F}_{\omega_{0}}$ is \emph{minimized} on its critical
points (compare section \ref{sub:Comparison-with-Donaldson's} and
the discussion in section 5 in \cite{bern2}). In the terminilogy
of \cite{don1} these latter critical points correpond to \emph{balanced
metrics.} Donaldson used his result, combined with the deep convergence
results in \cite{do1} for balanced metrics, in the limit when $L$
is replaced by a large tensor power, to prove a lower bound on Mabuchi's
K-energy functional. It will be shown in section \ref{sec:Convergence-towards-Mabuchi's}
how to deduce this latter result more directly from Theorem \ref{thm:main}
above.

It should also be pointed out that the inequality proved by Donaldson
corresponds to a \emph{lower} bound on $\mathcal{\mathcal{F}}_{\omega_{0}}(u)$
in the present setting, which however will depend on $u$ through
its volume form $(\omega_{u})^{n}/n!$ (see the end of section\ref{sub:Comparison-with-Donaldson's}
).

\subsection{Concerning the proof of Theorem \ref{thm:main}.}

The proof of Theorem \ref{thm:main} relies on the recent work \cite{bern2}
of Berndtsson combined with some global pluripotential theory developed
in \cite{b-b,b-d} (see also \cite{bbgz} for the case $L=-K_{X}$).
On one hand \cite{bern2} gives that $\mathcal{F}_{\omega_{0}}$ is
{}``geodesically'' convex wrt the Riemann metric on the space $\mathcal{H}_{\omega_{0}}$
introduced by Mabuchi \cite{m}. In turn, this fact is used to show
that any critical point maximizes $\mathcal{F}_{\omega_{0}}$ on $\mathcal{H}_{\omega_{0}},$
using the existence of (generalized) $C^{0}-$geodesics in the closure
$\overline{\mathcal{H}}_{\omega_{0}}.$ On the other hand, a main
point in the proof of the {}``uniqueness'' of critical points is
to show that there are no smooth extremal points of $\mathcal{\mathcal{F}}_{\omega_{0}}$
in the {}``boundary'' of $\mathcal{H}_{\omega_{0}},$ i.e. in $\overline{\mathcal{H}}_{\omega_{0}}-\mathcal{H}_{\omega_{0}}.$
Following \cite{b-b,bbgz} this is shown by extending $\mathcal{\mathcal{F}}_{\omega_{0}}$
to a (Gâteaux) differentiable function on all of $C^{0}(X),$ by replacing
$\mathcal{E}_{\omega_{0}}$ with the composed map $\mathcal{E}_{\omega_{0}}\circ P_{\omega_{0}},$
where $P_{\omega_{0}}$ is the following (non-linear) projection operator
from $C^{0}(X)$ onto $\mathcal{C}^{0}(X)\cap\overline{\mathcal{H}}_{\omega_{0}}:$
\begin{equation}
P_{\omega_{0}}[u](x)=\sup\left\{ v(x):\, v\in\mathcal{H}_{\omega_{0}},\,\, v\leq u\right\} \label{eq:proj oper intro}\end{equation}

\begin{rem}
\label{rem:finite energy}Consider the setting of Theorem \ref{thm:main}
and assume that there exists a (smooth) critical point, which we may
assume is given by $0.$ Then the inequality furnised by the theorem,
i.e. \[
\mathcal{F}_{\omega_{0}}(u):=\mathcal{E}_{\omega_{0}}(u)-\mathcal{L}_{\omega_{0}}(u)\leq0\]
actually holds for all $u$ in $\mathcal{E}^{1}(X,\omega_{0}),$ i.e.
for al $u$ in the convex set of all $u$ in $\overline{\mathcal{H}}_{\omega_{0}}$
with \emph{finite energy;} $\mathcal{E}(u)>-\infty,$ where \[
\mathcal{E}(u):=\inf_{u'\geq u}\mathcal{E}(u')\]
 when $u'$ ranges over all elements in $\mathcal{H}_{\omega_{0}}$
such that $u'\geq u.$ Equivalently, $\int_{X}(\omega_{u})^{n}=\mbox{Vol}(L)$
and $-\int_{X}u(\omega_{u})^{n}<\infty$ in terms of \emph{non-pluripolar
products} (see \cite{bbgz} and references therein). The inequality
on all of $\mathcal{E}^{1}(X,\omega_{0})$ is simply obtained by writing
$u$ as a decreasing limit of elements in $\mathcal{H}_{\omega_{0}}$
and using the continuity of $\mathcal{E}$ and $\mathcal{L}_{\omega_{0}}$
under such limits \cite{bbgz} (note that $e^{-u}$ is integrable
if $\mathcal{E}(u)>-\infty$ \cite{bbgz}).

Moreover, in the case when $\mbox{Aut}_{0}(X,L)$ is discrete it can
be shown that any maximizer of $\mathcal{F}_{\omega_{0}}$ on $\mathcal{E}^{1}(X,\omega_{0}),$
is in fact equal to a constant. The proof is a simple adaptation of
the argument in \cite{bbgz} concerning the case $L=-K_{X}.$ It would
be interesting to know if the general uniqueness statement in Theorem
\ref{thm:main} also remains true in the larger class $\mathcal{E}^{1}(X,\omega_{0})?$ 
\end{rem}
It is a pleasure to thank Bo Berntdsson for illuminating discussions
on the topic of the present paper, in particular in connection to
\cite{bern2}. It is an equal pleasure to thank Sébastien Boucksom,
Vincent Guedj and Ahmed Zeriahi for discussions and stimulation coming
from the colaboration \cite{bbgz}. The author is also grateful to
Yanir Rubinstein and Bálint Virág for helpful comments on a preliminary
version of this paper.

\subsection*{Organization}

In section 2 preliminaires for the proofs of the main results appearing
in section 3 are given. The proof of the uniqueness statement in the
main theorem relies on higher order regularity for {}``geodesics''
defined by inhomogenous Monge-Ampère equations. An alternative proof
based on considerably more elementary regularity results is given
in section \ref{sub:Alternative-proof-of}. In section \ref{sub:Arithemtic-applications}
applications to Arithmetic (Arakelov) geometry are briefly indicated.
In section 4 some of the previous results are interpreted in terms
of $SU(2)-$invariant determinantal random point process on $S^{2}.$
Finally, in section 5 the limit when the line bundle $L$ is replaced
by a large tensor power is studied and a new proof of the lower bound
on Mabuchi's $K-$energy for a polarized projective manifold is given.
Relations to Donaldson's work are also discussed. In the appendix
some formulas involving Bergman kernels are recalled and a {}``Bergman
kernel proof'' of Theorem \ref{thm:(Berndtsson)-Let-} is given.

\section{Preliminaries: Geodesics and energy functionals}

\subsection{Geodesics}

The infinite dimensional space $\mathcal{H}_{\omega}$ inherits an
\emph{affine} Riemannian structure from its natural imbedding as on
open set in $\mathcal{C}^{\infty}(X).$ Mabuchi, Semmes and Donaldson
(see \cite{ch} and references therein) introduced another Riemannian
structure on $\mathcal{H}_{\omega}$ (modolo the constants) defined
in the following way. Identifying the tangent space of $\mathcal{H}_{\omega}$
at the point $u$ with $\mathcal{C}^{\infty}(X)$ the squared norm
of a tangent vector $v$ at the point $u$ is defined as \[
\int_{X}v^{2}(\omega_{u})^{n}/n!.\]
However, the \emph{existence} of a geodesic $u_{t}$ in $\mathcal{H}_{\omega}$
connecting any given points $u_{0}$ and $u_{1}$ is an open and even
dubious problem. There are two problems: it is not known if $i)$
$u_{t}$ smooth, $ii)$ $\omega_{u_{t}}$ is strictly positive, as
a current. As is well-known such a geodesic may, if it exists, be
obtained as the solution of a homogenous Monge-Ampère equation (see
below). In the following we will simply take this characterization
as the \emph{definition} of a geodesic. It will also be important
to consider the larger space $\overline{\mathcal{H}}_{\omega_{0}}\cap C^{0}(X),$
since a priori the path $u_{t}$ may leave $\mathcal{H}_{\omega}.$
\begin{defn}
A continuous path in $\overline{\mathcal{H}}_{\omega_{0}}\cap C^{0}(X)$
$u_{t}$ will be called a \emph{$\mathcal{C}^{0}-$geodesic} connecting
$u_{0}$ and $u_{1}$ if $U(w,x):=u_{t}(x),$ where $t=\log\left|w\right|,$
is continuous on \[
M:=\{1\leq\left|w\right|\leq e\}\times X:=A\times X\]
with $dd^{c}U+\pi_{X}^{*}\omega_{0}\geq0$ and \begin{equation}
(dd^{c}U+\pi_{X}^{*}\omega_{0})^{n+1}=0\label{eq:dirichlet pr}\end{equation}
 in the interiour of $M$ in the sense of pluripotential theory \cite{g-z,de3},
where $\pi_{X}$ denotes the projection from $M$ to $X.$ 
\end{defn}
As shown in\cite{b-d,bbgz} $U(w,x)$ exists and is uniquely defined
as the extension from $\partial M$ obtained as the upper envelope
\begin{equation}
U(w,x)=\sup\left\{ V(w,x):\, V\in\mathcal{H}_{\pi_{X}^{*}\omega_{0}}(M),\,\, V\leq U\,\mbox{on}\,\partial M\right\} ,\label{eq:envelop}\end{equation}
where $\mathcal{H}_{\pi_{X}^{*}\omega_{0}}(M)$ denotes the set of
all smooth functions $V$ on $M$ such that $dd^{c}U+\pi_{X}^{*}\omega_{0}>0.$
If $u_{t}$ is such that $dd^{c}U+\pi_{X}^{*}\omega_{0}\geq0$ then
$u_{t}$ will be called a \emph{psh path} (or a \emph{subgeodesic}).
In local computations we will often make the identification $u_{t}(x)=U(w,x)$
extending $t$ to a complex variable. Then $u_{t}(x)$ is independent
of the imaginary part of $t$ and is hence \emph{convex} wrt real
$t.$ 

In the proof of the uniqueness part of Theorem \ref{thm:main} we
will have great use for the following regularity result for geodesics
in $\overline{\mathcal{H}}_{\omega_{0}},$ shown by Chen \cite{ch}.
See also \cite{bl} for a detailed analysis of the proof and some
refinements. The proof uses the method of continuity combined with
very precise a prioiri estimates on the perturbed Monge-Ampère equations.
\begin{thm}
\label{thm:(Chen)-Assume-that}(Chen) Assume that the boundary data
in the Dirichlet problem \ref{eq:dirichlet pr} for the Monge-Ampère
operator on $M$ is smooth on $\partial M.$ Then $U\in\mathcal{C}_{\C}^{1,1}(M).$
More precisely, the mixed second order complex derivatives of $U$
are uniformly bounded, i.e. there is a positive constant $C$ such
that \[
0\leq(dd^{c}U+\pi_{X}^{*}\omega_{0})\leq C(\pi_{X}^{*}\omega_{0}+\pi_{A}^{*}\omega_{A})\]
 where $\omega_{A}$ is the Eucledian metric on $A.$
\end{thm}
In the statement above we have used the (non-standard) notation $\mathcal{C}_{\C}^{1,1}(M)$
for the set of all functions $U$ such that, locally, the current
$dd^{c}U$ has coefficents in $L^{\infty}.$ Such a $U$ is called
\emph{almost $\mathcal{C}^{1,1}$} in \cite{bl}. Note that if $U\in\mathcal{H}_{\pi_{X}^{*}\omega_{0}}(M)$
then this is equivalent to $U$ having a bounded Laplacian $\Delta_{M}U,$
where $\Delta_{M}$ is the Laplacian on $M$ wrt the Kähler metric
$\pi_{X}^{*}\omega_{0}+\pi_{A}^{*}\omega_{A}$ on $M.$ As will be
explained in section \ref{sub:Alternative-proof-of} the proof of
the uniqueness statement in Theorem \ref{thm:main} may actually be
obtained by only using the bounds on the derivatives of $u_{t}$ on
$X$ for $t$ fixed. As shown very recently in \cite{b-d} such bounds
may be obtained by working directly with the envelope \ref{eq:envelop}. 
\begin{thm}
\label{thm:berman-dem}Assume that the boundary data in the Dirichlet
problem \ref{eq:dirichlet pr} for the Monge-Ampère operator on $M$
is in $\mathcal{C}^{1,1}(\partial M).$ Then $u_{t}\in\mathcal{C}_{\C}^{1,1}(X).$
More precisely, the mixed second order complex derivatives of $u_{t}$
on $X$ are uniformly bounded, i.e. there is a positive constant $C$
such that \[
0\leq(dd^{c}u_{t}+\omega_{0})\leq C\omega_{0}\]
 on $X.$
\end{thm}
One of the virtues of this latter approach is that the proof is remarkly
simple when $X$ is homogenous.

\subsection{The functional $\mathcal{L}_{\omega_{0}}.$ }

First note that the functional $\mathcal{L}_{\omega_{0}}(u)$ defined
by formula \ref{eq:l as toeplitz det intro} is increasing on $\mathcal{C}^{0}(X),$
wrt the usual order relation. This is an immediate consequence of
the basic geometric interpretation in \cite{b-b} of $\mathcal{L}_{\omega_{0}}(u)$
as propopertional to the logarithmic volume of the unit-ball in the
Hilbert space $H^{0}(X,L+K_{X})$ equipped with the Hermitian product
induced by the weight $\psi_{0}+u.$ Alternatively, it follows from
formula \ref{eq:deriv of l} below which shows that the differential
of the functional $\mathcal{L}_{\omega_{0}}$ on $\mathcal{C}^{0}(X)$
may be represented by the positive measure $\beta_{u}.$ Integrating
$\beta_{u}$ along a line segment in $\mathcal{C}^{0}(X)$ equipped
with its affine structure then shows that $\mathcal{L}_{\omega_{0}}(u)$
is increasing. 

The \emph{differential} of the functional $\mathcal{L}_{\omega_{0}}$
on $\mathcal{C}^{0}(X)$ is given by 

\begin{equation}
(d\mathcal{L}_{\omega_{0}})_{u}=\beta_{u},\label{eq:deriv of l}\end{equation}
in the sense that given any smooth function $v$ we have that \[
d(\mathcal{L}_{\omega_{0}}(u+tv))/dt_{t=0}=\int_{X}\beta_{u}v,\]
 where $\beta_{u}$ is the \emph{Bergman measure associated to $u.$}
This latter measure is the positive measure on $X$ defined as \begin{equation}
\beta_{u}=(i^{n^{2}}\frac{1}{N}\sum_{i=1}^{N}s_{i}\wedge\bar{s_{i}}e^{-\psi_{0}})e^{-u}\label{eq:bergman meas}\end{equation}
in terms of any given orthonormal base $(s_{i})$ in the Hilbert space
$H^{0}(X,L+K_{X})$ equipped with the Hermitian product induced by
the weight $\psi_{0}+u$ (compare section \ref{sub:Bergman-kernels}).
In particular this means that $\beta_{u}$ may be represented as $e^{-u}$
times a strictly positive smooth measure on $X$ if $L+K_{X}$ is
globally generated. The proof of formula \ref{eq:deriv of l} follows
more or less directly from the definition (see \cite{bern2} for a
geometric argument).

The following theorem, which is direct consequence of a result of
Berndtsson about the curvature of direct image bundles \cite{bern2},
considers the \emph{second} derivatives of $\mathcal{L}_{\omega_{0}}$
along a psh path. As a courtesy to the reader a proof of the theorem,
using Bergman kernels, is given in the appendix.
\begin{thm}
\label{thm:(Berndtsson)-Let-}(Berndtsson) Let $u_{t}$ be a continuous
psh path in $\mathcal{H}_{\omega_{0}}.$ Then the function $t\mapsto\mathcal{L}_{\omega_{0}}(u_{t})$
is convex. Moreover, if $\mathcal{L}_{\omega_{0}}(u_{t})$ is affine
and $u_{t}$ is a \emph{smooth} psh path with $\omega_{u_{t}}>0$
on $X$ for all $t,$ then there is an automorphism $S_{1}$ of $(X,L)$,
homotopic to the identity, such that $u_{1}-u_{0}=S_{1}^{*}\psi_{0}-\psi_{0}.$ 
\end{thm}
The convexity statement in \cite{bern2} assumed in fact that $u_{t}$
be\emph{ smooth.} However, by uniform approximation the convexity
statement above in fact holds for any \emph{continuous} psh path in
$C^{0}(X).$ Indeed, if $u_{t}$ is such a path, then there exists,
for example by Richbergs's approximation theorem \cite{d1}, a sequence
$U^{j}$ converging uniformly towards $U$ on $M$ such that $dd^{c}U^{j}+\pi^{*}\omega_{0}>0.$
Applying the theorem above to each $U^{j}$ and letting $j$ tend
to infinity then gives that $f(t):=\mathcal{L}_{\omega_{0}}(u_{t})$
is a uniform limit of convex functions and hence convex, proving the
claim. 

However, for the uniqueness statement the argument in \cite{bern2}
seems to require that $\omega_{u_{t}}$ be reasonably smooth in $(t,x).$
Moreover, the assumption that $\omega_{t}>0$ is crucial to be able
to define the vector fields $V_{t}$ that integrate to the automorphism
$S_{1}$ (see formula \ref{eq:int multip}).

\subsection{The functional $\mathcal{E}_{\omega_{0}}$}

First recall the following well-known formula for the differential
of the energy functional $\mathcal{E}_{\omega_{0}}$ defined by formula
\ref{eq:def of e intro}:

\begin{equation}
(d\mathcal{E}_{\omega_{0}})_{u}=\omega_{u}^{n}/n!\label{eq:deriv of e}\end{equation}
The following generalization from \cite{b-b} of the previous formula
to the functional $\mathcal{E}_{\omega_{0}}\circ P_{\omega_{0}},$
where $P_{\omega_{0}}$ is the non-linear projection \ref{eq:proj oper intro},
will be crucial for the proof of Theorem \ref{thm:main}:
\begin{thm}
\label{thm:deriv of composed }The functional $\mathcal{E}_{\omega_{0}}\circ P_{\omega_{0}}$
is Gâteaux differentiable on \textup{$\mathcal{C}^{0}(X).$} Its differential
at the \textup{point $u$ is represented by the measure $\omega_{P_{\omega_{0}}u}^{n}/n!,$
i.e.} given $u,v\in\mathcal{C}^{0}(X)$ the function $\mathcal{E}_{\omega_{0}}P_{\omega_{0}}(u+tv)$
is differentiable on $\R_{t}$ and \begin{equation}
d\mathcal{E}_{\omega_{0}}P_{\omega_{0}}(u+tv)/dt_{t=0}=\int_{X}v\omega_{P_{\omega_{0}}u}^{n}/n!\label{eq:deriv of comp}\end{equation}

\end{thm}
As for the second derivatives of $\mathcal{E}_{\omega_{0}}$ we have
the following Proposition which is well-known (at least in the smooth
case):
\begin{prop}
\label{pro:e is affine}The following properties of $\mathcal{E}_{\omega_{0}}$
hold:\end{prop}
\begin{itemize}
\item The functional $\mathcal{E}_{\omega_{0}}$ on $\overline{\mathcal{H}}_{\omega_{0}}\cap\mathcal{C}^{0}(X)$
is concave wrt the affine structure on $\mathcal{C}^{0}(X).$
\item Let $u_{t}$ be a $\mathcal{C}^{0}$- geodesic in $\overline{\mathcal{H}}_{\omega_{0}}$
connecting $u_{0}$ and $u_{1}.$ Then the functional $t\mapsto\mathcal{E}_{\omega_{0}}(u_{t})$
is affine and continuous on $[0,1].$\end{itemize}
\begin{proof}
(A proof also appears in \cite{bbgz}). Recall the following well-known
formula (see for example \cite{b-b}): \begin{equation}
d_{t}d_{t}^{c}\mathcal{E}_{\omega_{0}}(u_{t})=t_{*}(dd^{c}U+\pi^{*}\omega_{0})^{n+1}/(n+1)!,\label{eq:genaral second deriv}\end{equation}
 where $t_{*}$ denotes the natural push-forward map from $M$ to
$\C_{t}.$ In particular, setting $u_{t}=u_{0}+tu$ gives for real
$t$ $d^{2}\mathcal{E}_{\omega_{0}}(u_{t})/d^{2}t=-\int_{X}\left|\partial u\right|^{2}\omega_{0}^{n}\leq0$
(compare formula \ref{eq:expanding monge as c}) which proves the
first point of the proposition when $u$ is smooth. To handle the
general case one takes $u_{j}$ in $\mathcal{H}_{\omega_{0}}$ converging
uniformly to $u$ and uses that, according to Bedford-Taylor's classical
results, $\mathcal{E}_{\omega_{0}}$ is continuous under uniform limits
in $\overline{\mathcal{H}}_{\omega_{0}}\cap\mathcal{C}^{0}(X)$ (see
also \cite{b-b}). This shows that $\mathcal{E}_{\omega_{0}}(u_{t})$
is the limit of concave functions and hence concave. To prove the
last point take a sequence $U^{j}$ converging uniformly to $U$ on
$M$ and such that $dd^{c}U^{j}+\pi^{*}\omega_{0}>0$ (compare the
discussion below Theorem \ref{thm:(Berndtsson)-Let-}). By Bedford-Taylor
$(dd^{c}U^{j}+\pi^{*}\omega_{0})^{n+1}$ tends weakly to $(dd^{c}U^{j}+\pi^{*}\omega_{0})^{n+1}$
in the interiour of $M.$ Hence, formula \ref{eq:genaral second deriv}
shows that the second real derivatives of $\mathcal{E}_{j}(t):=\mathcal{E}(u_{t}^{j})$
tend weakly to zero in the sense of distributions for $t\in]0,1[.$
But since the sequence $\mathcal{E}_{j}(t)$ of smooth convex functions
tends to $\mathcal{E}(t)$ it follows that $\mathcal{E}(t)$ is affine
on $]0,1[$ and hence by continuity on all of $[0,1].$ To be more
precise: since $U$ is continuous on the compact set $M$ the family
$u_{t}$ tends to $u_{0}$ and $u_{1}$ uniformly when $t\rightarrow0$
and $t\rightarrow1,$ respectively. Finally, since $\mathcal{E}$
is continuous under uniform limits in $\overline{\mathcal{H}}_{\omega_{0}}\cap\mathcal{C}^{0}(X)$
this proves that $\mathcal{E}$ is continuous up to the boundary on
$[0,1].$
\end{proof}
Before turning to the proof of Theorem \ref{thm:main}, we recall
the following basic cocycle property of the functional $\mathcal{F}_{\omega_{0}}:=\mathcal{E}_{\omega_{0}}-\mathcal{L}_{\omega_{0}}:$
\begin{equation}
\mathcal{F}_{\omega_{u_{2}}}(u_{1})+\mathcal{F}_{\omega_{u_{3}}}(u_{2})=\mathcal{F}_{\omega_{u_{3}}}(u_{1}),\label{eq:cocy}\end{equation}
 which is a direct consequences of the corresponding cocyle properties
of $\mathcal{E}_{\omega_{0}}$ and $\mathcal{L}_{\omega_{0}}.$ These
latter properties in turn are immediately obtained by integrating
the corresponding differentials along line segments (compare \cite{ti}).
\begin{rem}
\label{rem:j-f}The funtional $\mathcal{E}_{\omega_{0}}$ may be expressed
in terms of a generalized Dirichlet type energy $J_{\omega_{0}}:$
\[
-\mathcal{E}_{\omega_{0}}(u)=J_{\omega_{0}}(u)-\frac{1}{V}\int u\omega_{0},\]
 where $J_{\omega_{0}}$ is Aubin's energy functional \begin{equation}
J_{\omega_{0}}(u):=\frac{1}{V}\sum_{i=1}^{n-1}\frac{i+1}{n+1}\int du\wedge du^{c}\wedge(\omega_{0})^{i}\wedge(\omega_{u})^{n-1-i}\label{eq:def of j}\end{equation}
(compare \cite{ti} p. 58). Note that if $n=1$ then $J_{\omega_{0}}$
is non-negative for \emph{any} $u,$ while the natural condition to
obtain non-negativity when $n>1$ is that $\omega_{u}\geq0.$ On the
other hand as shown in \cite{rub0} (lemma 2.1), there are examples
of general smooth $u$ with $J_{\omega_{0}}<0$ for any manifold $X$
of dimension $n>1.$ As a direct consequence it was shown in \cite{rub0},
in the case $L=-K_{X},$ that any such fucntion $u$ violates the
inequality in Theorem \ref{thm:main}. A similar argument applies
to a homogeneous line bundle $L$ as in Corollary \ref{cor:homeg}.
Indeed, without affecting the value of $J_{\omega_{0}}(u)$ we may
assume that $\int_{X}u\omega_{0}^{n}=0$ so that $-\mathcal{E}_{\omega_{0}}(u)=J_{\omega_{0}}(u)<0.$
Now, using the notation of section \ref{sec:Application-to-determinantal}
below, \[
-\mathcal{L}_{\omega_{0}}(u)=\log\E_{N}(e^{-(u(x_{1})+...+u(x_{n})})\geq-\E_{N}(u(x_{1})+...+u(x_{n})),\]
using Jensen's inequality in the last step. Moreover, by formula \ref{eq:beta as one-pt}
$\E_{N}(u(x_{1})+...+u(x_{n}))=\int_{X}u\beta_{u}).$ Since $\beta_{u}=\omega_{0}^{n}/V$
in the homogenous case (compare the proof of Corollary \ref{cor:homeg}),
this means that $-\mathcal{L}_{\omega_{0}}(u)\geq0.$ Hence, $u$
violates the inequality referred to above.
\end{rem}

\section{Proofs of the main results}

\subsection{\label{sub:Proof-of-Theorem-main}Proof of Theorem \ref{thm:main}}

By the cocycle property of $\mathcal{F}_{\omega_{0}}$ (see \cite{bbgz,b-b};
it is shown by integrating the differential of $\mathcal{F}_{\omega_{0}}$
along line segments) we may without loss of generality assume that
$u=0$ is critical. Take a \emph{continuous} element $u_{1}$ in $\mathcal{H}_{\omega_{0}}$
and the corresponding $\mathcal{C}^{0}-$geodesic $u_{t}$ connecting
$u_{0}=0$ and $u_{1}.$ Since $u_{t}$ is a continuous path, combining
Theorem \ref{thm:(Berndtsson)-Let-} and Proposition \ref{pro:e is affine}
gives that $\mathcal{F}_{\omega_{0}}(t):=\mathcal{F}_{\omega_{0}}(u_{t})$
is a continous concave function on $[0,1].$ Hence, the inequality
in Theorem \ref{thm:main} will follow once we have shown that \begin{equation}
\frac{d}{dt}_{t=0+}\mathcal{F}(u_{t})\leq0.\label{eq:ineq for right der}\end{equation}
Of course, if $u_{t}$ were known to be a \emph{smooth} path then
this would be an immediate consequence of the assumption that $u_{0}$
is critical combined with the chain rule (which would even yield equality
above). To prove \ref{eq:ineq for right der} first observe that by
the concavity in Prop \ref{pro:e is affine}\[
(\mathcal{E}_{\omega_{0}}(u_{t})-\mathcal{E}_{\omega_{0}}(u_{0}))/t\leq\frac{1}{t}\int_{X}(u_{t}-u_{0})(\omega_{u_{0}})^{n}/n!\]
Hence, the monotone convergence theorem applied to the sequence $(u_{t}-u_{0})/t$
which decreases to the right derivative $v_{0}$ of $u_{t}$ at $t=0$
(using that $u_{t}$ is convex in $t)$ gives \begin{equation}
\frac{d}{dt}_{t=0+}\mathcal{E}_{\omega_{0}}(u_{t})\leq\int_{X}v_{0}(\omega_{u_{0}})^{n}/n!\label{eq:upper bound on deriv of energy}\end{equation}
Hence, \[
\frac{d}{dt}_{t=0+}\mathcal{F}(u_{t})\leq\int_{X}((\omega_{u_{0}})^{n}/n!-\beta_{u_{u}})v_{0}=0,\]
 where we have also used the dominated convergence theorem to differentiate
$\mathcal{L}_{\omega_{0}}(u_{t})$ (compare \cite{bbgz,b-b}). This
finishes the proof of \ref{eq:ineq for right der}and hence the first
statement in the theorem follows.

\emph{Uniqueness:} Assume now that $u_{1}$ is a smooth maximizer
of $\mathcal{F}_{\omega_{0}}$ on $\overline{\mathcal{H}}_{\omega_{0}}$
i.e. that $\mathcal{F}_{\omega_{0}}(u_{1})=\mathcal{F}_{\omega_{0}}(u_{0})$
by the previous step. Since $\mathcal{F}_{\omega_{0}}(t):=\mathcal{F}_{\omega_{0}}(u_{t})$
is continuous and concave it follows that $u_{t}$ maximizes $\mathcal{F}_{\omega_{0}}$
on $\overline{\mathcal{H}}_{\omega_{0}}\cap\mathcal{C_{\C}}^{1,1}(X)$
for all $t.$ Next, we will show that $u_{t}$ satisfies the Euler-Lagrange
equation \ref{eq:e-l for f} for any fixed $t$ (see \cite{bbgz}
for similar arguments). To this end fix $t=t_{0}$ and set $u_{t_{0}}:=u.$
Given a smooth function $v$ on $X$ consider the function $f(t):=\mathcal{E}_{\omega_{0}}(P_{\omega_{0}}(u+tv))-\mathcal{L}_{\omega_{0}}(u+tv)$
on $\R_{t}.$ Since, the functional $\mathcal{L}_{\omega_{0}}$ is
increasing on $\mathcal{C}^{0}(X)$ we have $f(t)\leq\mathcal{F}_{\omega_{0}}(P_{\omega_{0}}(u+tv)).$
By assumption this means that the maximal value of the function $f(t)$
is attained for $t=0$ (also using that $P_{\omega_{0}}u=u).$ In
particular, since by Theorem \ref{thm:deriv of composed } $f(t)$
is differentiable $df/dt=0$ at $t=0$ and Theorem \ref{thm:deriv of composed }
and formula \ref{eq:deriv of l} hence show that the Euler-Lagrange
equation \ref{eq:e-l for f} holds (since it holds when tested on
any smooth function $v).$ 

Next, we will prove that $U\in\mathcal{C}^{\infty}(\dot{M}),$ where
$\dot{M}$ denotes the interiour of $M.$ By Theorem \ref{thm:(Chen)-Assume-that}
$U$ is in $\mathcal{C}_{\C}^{1,1}(M).$ Moreover, by the homogenous
Monge-Ampère equation \ref{eq:dirichlet pr} and the Euler-Lagrange
equation \ref{eq:e-l for f} we have \[
(dd^{c}(U+\left|w\right|^{2})+\pi_{X}^{*}\omega_{0})^{n+1}=i\beta_{u}\wedge dw\wedge d\bar{w}\]
 Hence, the following equation holds locally on $\C^{n+1}$ (where
we for simplicity have kept the notation $U$ for the function obtained
after subtracting a smooth and hence harmless function from $U):$
\begin{equation}
\det(\partial_{\zeta_{i}}\partial_{\bar{\zeta}_{j}}U)=e^{-U}\rho,\label{eq:local exp ma}\end{equation}
 where $\rho$ is a positive smooth function, depending on $U$ (compare
the discussion below formula \ref{eq:bergman meas}). In particular,
$\det(\partial_{\zeta_{i}}\partial_{\bar{\zeta}_{j}}U)$ is locally
in $\mathcal{C}_{\C}^{1,1}.$ But then Theorem 2.5 in \cite{bl0},
which is a complex analog of a result of Trudinger for fully non-linear
elliptic operators (compare Evans-Krylov theory), gives that $U$
is locally in the Hölder space $\mathcal{C}^{2,\alpha}$ for some
$\alpha>0.$ Now the equation \ref{eq:local exp ma} shows that $\det(\partial_{\zeta_{i}}\partial_{\bar{\zeta}_{j}}U)$
is also in $\mathcal{C}^{2,\alpha}.$ Finally, since we have hence
shown that $U\in C^{2},$ standard theory of uniformly elliptic operators
then allows us to boot strap using \ref{eq:local exp ma} and deduce
that $U\in\mathcal{C}^{\infty}$ locally (see Theorem 2.2 in \cite{bl}).
Note also that by the Euler-Lagrange equation \ref{eq:e-l for f}
we have a uniform lower bound $\omega_{u_{t}}^{n}>\delta\omega_{0}^{n}$
(also using the lower bound in formula \ref{eq:bounds on beta} in
the appendix). Combining the previous lower bound with the upper bound
$\omega_{u_{t}}\leq C\omega_{0}$ from Theorem \ref{thm:(Chen)-Assume-that}
then shows that there is a positive constant $C',$independent of
$t,$ such that \begin{equation}
1/C'\omega_{0}\leq\omega_{u_{t}}\leq C'\omega_{0}\label{eq:bound on omega t}\end{equation}
Since, by the above arguments $\mathcal{F}_{\omega_{0}}(u_{t})$ and
$\mathcal{E}_{\omega_{0}}(u_{t})$ are both affine (and even constant)
it follows that $\mathcal{L}_{\omega_{0}}(u_{t})$ is affine. In case
$U$ were smooth \emph{up to the boundary} of $M$ applying Theorem
\ref{thm:(Berndtsson)-Let-} would hence prove the uniqueness statement
in Theorem \ref{thm:main}. To prove the general case we may without
loss of generality assume that $u_{t}(x)$ is smooth on $[0,1[\times X$
(otherwise we just apply the same argument on $[1/2,1[$ and $]0,1/2]).$
For any $\epsilon>0$ Theorem \ref{thm:(Berndtsson)-Let-} (see Theorem
2.6 in \cite{bern2}) furnishes a 1-parameter holomorphic family $S_{t}$
in $\mbox{Aut}_{0}(X,L)$ with $t\in[0,1-\epsilon]$ defined by the
ordinary differential equation \begin{equation}
\frac{dS_{t}(x(t))}{dt}=d_{X}(S(x(t))[V_{t}]_{x(t)}\label{eq:ode}\end{equation}
with the iniatial data $S_{0}=I$ (the identity), where $V_{t}$ is
the vector field on $X$ of type $(1,0)$ defined by the equation
\begin{equation}
\omega_{u_{t}}(V_{t},\cdot)=\bar{\partial}_{X}(\partial_{t}u),\label{eq:int multip}\end{equation}
 where $\bar{\partial}_{X}$ is the $\bar{\partial}-$operator on
$X$ and $\partial_{t}$ is the partial holomorphic derivative wrt
$t$ for $z$ fixed in $X.$ As shown in \cite{bern2} the fact that
$\mathcal{L}(u_{t})$ is affine wrt $t$ forces the vector field $V_{t}$
to be holomorphic on $X$ for each $t$ and it then follows that $V_{t}$
is holomorphic wrt $t$ as well (a slight variant of this argument
is recalled in section \ref{sub:Alternative-proof-of}). Furthermore,
as shown in \cite{bern2} \begin{equation}
\psi_{t}-S_{t}^{*}\psi_{0}=C_{t}\label{eq:action on weights}\end{equation}
 where $\psi_{t}=\psi_{0}+u_{t}$ and $C_{t}$ is a constant for each
$t,$ i.e. \begin{equation}
\omega_{u_{t}}=S_{t}^{*}\omega_{0}.\label{eq:action on curvature}\end{equation}
Now, by the bound \ref{eq:bound on omega t} on $\omega_{u_{t}}$
the point-wise norm of the vector field $V_{t}$ wrt the metric $\omega_{0}$
is uniformly bounded in $t$ on all of $X.$ Hence, the equation \ref{eq:ode}
and a basic normal families argument applied to the family $S_{t}$
yields a subsequence $S_{t_{j}}$ and a holomorphic map $S_{1}$ on
$X$ such that $S_{t_{j}}(x)\rightarrow S_{1}(x)$ uniformly on $X$
(wrt the distance defined by the metric $\omega_{0})$ where $S_{1}$
is a biholomorphism according to the relation \ref{eq:action on curvature}.
Finally, letting $t_{j}\rightarrow1$ in the relation \ref{eq:action on weights}
and using that $u_{t}$ is continuous on $[0,1]\times X$ finishes
the proof of the uniqueness statement in the theorem. 
\begin{rem}
It was not explicetly pointed out in \cite{bern2} that $S_{t}$ lifts
to $L,$ but this fact follows from lemma 12 in \cite{do1}. 
\end{rem}

\subsection{Proof of Corollary \ref{cor:homeg}}

First observe that we may assume that $H^{0}(X,L+K_{X})$ has a non-zero
element (otherwise the corrollary is trivally true). But since $(X,L)$
is homogenuous it then follows immediately that $L+K_{X}$ is globally
generated. Hence, the conditions in Theorem \ref{thm:main} are satisfied.

Assume now that $\omega_{0}$ is invariant under the holomorphic and
transitive action of $K$ on $X.$ Then it follows that $0$ is a
critical point. Indeed, the volume form $\omega_{0}^{n}/n!$ is invariant
under the action of $K$ on $X$ and so is the Bergman measure $\beta(0)$
(since it is defined in terms of the $K-$invariant weight $\psi_{0}).$
Since the action of $K$ is transitive and both measures are normalized
it follows that the function $(\omega_{0}^{n}/n!)/\beta(0)$ on $X$
is constant and hence equal to one. In other words, $0$ is a critial
point and by Theorem \ref{thm:main} the inequality in the statement
of Corollary \ref{cor:homeg} then holds. Finally, the last statement
of the corollary is a direct consequence of the uniqueness part of
Theorem \ref{thm:main}.

\subsection{Proof of Corollary \ref{cor:moser}}

Let us first prove the first statement of the corollary. Since $\mathcal{C}^{\infty}(X)$
is dense in $W^{1,2}(X)$ we may assume that $u$ is smooth. First
observe that \begin{equation}
\mathcal{F}_{\omega_{0}}(u)\leq\mathcal{F}_{\omega_{0}}(P_{\omega_{0}}u).\label{eq:pf of cor mos}\end{equation}
 To see this note that, since, by definition, $P_{\omega_{0}}u\leq u$
the fact that $\mathcal{L}_{\omega_{0}}$ is increasing immediately
implies $\mathcal{L}_{\omega_{0}}(u)\geq\mathcal{L}_{\omega_{0}}(P_{\omega_{0}}u).$
Next, observe that by the cocycle property of $\mathcal{F}_{\omega_{0}}(u)$
\[
\mathcal{E}_{\omega_{0}}(u)=\mathcal{E}_{\omega_{0}}(P_{\omega_{0}}u)+\int_{X}(u-P_{\omega_{0}}u)(\omega_{u}+\omega_{P_{\omega_{0}}u})/2\]
 But, since, as is well-known the measure $\omega_{P_{\omega_{0}}u}$
is supported on the open set $\{u>P_{\omega_{0}}u\}$ (cf. Prop. 1.10
in \cite{b-b} for a generalization) we have that the last term above
is equal to \[
\int_{X}(u-P_{\omega_{0}}u)(\omega_{u}-\omega_{P_{\omega_{0}}u})/2=\int_{X}(u-P_{\omega_{0}}u)(dd^{c}(u-P_{\omega_{0}}u)=\]
\[
=-\int_{X}d(u-P_{\omega_{0}}u)\wedge d^{c}(u-P_{\omega_{0}}u)\leq0,\]
 where we have integrated by parts in the last equality, which is
justified since, for example, by Theorem \cite{be1} $P_{\omega_{0}}u$
is in $\mathcal{C}^{1,1}(X)$ (but using that $P_{\omega_{0}}u$ is
in $\mathcal{C}^{0}(X)$ is certainly enough by classical potential
theory). Hence, $\mathcal{E}_{\omega_{0}}(u)\leq\mathcal{E}_{\omega_{0}}(P_{\omega_{0}}u)$
which finishes the proof of \ref{eq:pf of cor mos}. Since, $\omega_{P_{\omega_{0}}u}\geq0$
uniform approximation let's us apply Corollary \ref{cor:homeg} to
deduce \[
\mathcal{F}_{\omega_{0}}(u)\leq\mathcal{F}_{\omega_{0}}(P_{\omega_{0}}u)\leq0\]
 which proves the first statement of the corollary. 

Finally, the uniqueness will follow from Corollary \ref{cor:homeg}
once we know that a maximizer $u$ of $\mathcal{F}_{\omega_{0}}$
on $W^{1,2}(S^{2})$ is smooth with $\omega_{u}>0.$ By the previous
step we may assume that $\omega_{u}\geq0.$ But since $W^{1,2}(S^{2})$
is a linear space containing $\mathcal{C}^{\infty}(X)$ the Euler-Lagrange
equations $\omega_{u_{0}}+dd^{c}u=\beta(u)$ hold for the maximizer
$u.$ Since $\beta(u)=e^{-u}\rho>0$ with $\rho$ smooth, local elliptic
estimates for the Laplacian then show that $u$ is in fact smoth with
$\omega_{t}>0.$ All in all we have proved that \begin{equation}
-\mathcal{L}_{k\omega_{0}}(u)\leq-\mathcal{E}_{k\omega_{0}}(u)\label{eq:abs ineq in pf cor fang}\end{equation}
 for $L=k\mathcal{O}(1)$ with conditions for equality.

\subsubsection{\label{sub:Explicit-expression}Explicit expression}

To make the previous inequality more explicit note that, by definition,
\[
\mathcal{E}_{k\omega_{0}}(u):=\frac{1}{2\int k\omega_{0}}\int(udd^{c}u+u2k\omega_{0})=\frac{1}{2k}\int udd^{c}u+\int u\omega_{0})\]
 Moreover, since for $X=\P^{1}$ we have $K_{X}=-\mathcal{O}(2)$
it follows that $L+K_{X}=\mathcal{O}(k-2)=:\mathcal{O}(m).$ Under
this identification the scalar product on $H^{0}(X,L+K_{X})$ may
be written as \[
\left\langle s,t\right\rangle _{k\psi_{0}+u}=c\int s\bar{t}e^{-(u+m\psi_{0})}\omega_{0}\]
 using that $\omega_{0}$ is a Kähler-Einstein metric, i.e. $\omega_{0}(z):=dd^{c}\psi_{0}=ce^{-2\psi_{0}}idz\wedge d\bar{z}$
for some numerical constant $c.$ Since the funtional $\mathcal{L}$
is invariant under on overall scaling in definition of the scalar
product $\left\langle \cdot,\cdot\right\rangle _{\psi}$ we may as
well assume that $c=1.$ Hence, since $N_{m}=m+1,$ we have\begin{equation}
\mathcal{L}_{m}(u):=(m+1)\mathcal{L}_{\omega_{0,k}}(u)=-\log\mbox{det}(c_{i}c_{j}\int_{\C}\frac{z^{i}\bar{z}^{j}}{(1+z\bar{z})^{m}}e^{-u}\omega_{0}),\label{eq:def of l m in pf fang}\end{equation}
 where $c_{i}=(\int\frac{\left|z^{i}\right|^{2}}{(1+z\bar{z})^{m}}\omega_{0})^{-1/2}.$
Hence, the inequality \ref{eq:abs ineq in pf cor fang} may be expressed
as \begin{equation}
\log\mbox{det}(c_{i}c_{j}\int_{\C}\frac{z^{i}\bar{z}^{j}}{(1+z\bar{z})^{m}}e^{-u}\omega_{0})\leq-\frac{m+1}{(m+2)}\frac{1}{2}\int(udd^{c}u)-(m+1)u\omega_{0}.\label{eq:ineq for l m in pf fang}\end{equation}
 In particular, when $m=0$ the inequality above reads \[
\log(\int_{S^{2}}e^{-u}\omega_{0})\leq\frac{1}{4}\int(udd^{c}u)+\int u\omega_{0}).\]
Finally, to compare with the notation of Onofri \cite{on}, note that,
by definition, $dd^{c}u=\frac{i}{2\pi}\partial\bar{\partial u}$ and
hence, integration by parts gives, \[
-\int udd^{c}u=\frac{1}{\pi}\frac{i}{2}\int\partial u\wedge\bar{\partial u}.\]
Moreover, in terms of a given local holomorphic coordinate $z=x+iy,$
we have $\frac{i}{2}\partial u\wedge\bar{\partial u}=\frac{1}{4}\left|\nabla u\right|^{2}dx\wedge dy$,
where $\nabla=(\partial_{x},\partial_{y})$ is the gradient wrt the
local Euclidian metric. By conformal invariance we hence obtain $-\int udd^{c}u=\frac{1}{4\pi}\int\left|\nabla u\right|^{2}d\mbox{Vol}_{g}$
for\emph{ any} Riemannian metric $g$ on $S^{2}$ conformally equivalent
to $g_{0}.$ In particular, taking $g$ as the usual round metric
on $S^{2}$ induced by its embedding as the unit-sphere in Eucledian
$\R^{3}$ finally gives \[
\log(\int_{S^{2}}e^{-u}d\mbox{Vol}{}_{g}/4\pi)\leq\frac{1}{4}\int\left|\nabla u\right|^{2}-u)d\mbox{Vol}{}_{g}/4\pi),\]
 using that $\omega_{0}=d\mbox{Vol}{}_{g}/4\pi.$ This is precisely
the inequality proved by Onofri \cite{on}.
\begin{rem}
\label{rem:sharp}The inequality in Corollary \ref{cor:moser} is
not only sharp in the sense that it is saturated for \emph{some} function
(for example $u=0)$, but also in the sense that if there exist constants
$A,B$ with $B\geq0$ such that \begin{equation}
-\mathcal{L}_{m}(u)\leq-A\int_{S^{2}}u\omega_{0}+B\int du\wedge d{}^{c}u,\label{eq:ansats}\end{equation}
 for all smooth $u,$ then $A=m+1$ and $B\geq\frac{m+1}{(m+2)}\frac{1}{2}.$
Indeed, by the conditions for equality in Corollary \ref{cor:moser}
we may find a function $u,$ which is not identically constant, saturating
the inequality in Corollary \ref{cor:moser}. After adding a suitable
constant to $u$ we may assume that $\int_{S^{2}}u\omega_{0}=0$ and
hence that \[
-\mathcal{L}_{m}(u)=\frac{m+1}{(m+2)}\frac{1}{2}\int_{S^{2}}du\wedge d{}^{c}u.\]
 Now using \ref{eq:ansats} it follows that \[
\frac{m+1}{(m+2)}\frac{1}{2}\int_{S^{2}}du\wedge d{}^{c}u\leq B\int_{S^{2}}du\wedge d{}^{c}u,\]
 i.e. that $B\geq\frac{m+1}{(m+2)}\frac{1}{2}.$ Next, taking $u$
as a constant $c$ in \ref{eq:ansats} gives \[
-c(m+1)\leq-cA\]
 But since $c$ was arbitrary it follows that $A=m+1.$ In fact, a
variant of the previous argument shows that Corollary \ref{cor:homeg}
is sharp in a similar sense (by replacing $\int du\wedge d{}^{c}u/2$
with Aubin's $J-$functional \ref{eq:def of j}). The details are
omitted.
\end{rem}

\subsection{Proof of Corollary \ref{cor:max of det}}

We keep the notation from the previous section. For simplicity we
will write $\Delta_{u}:=\Delta_{\bar{\partial}_{u}}^{(m)}.$ Following
\cite{g-s} we will first express $\det\Delta_{u}$ in terms of $-\mathcal{L}_{m}(u).$
Since the dimension $h_{1}(\mathcal{O}(m))$ of the first Dolbeault
cohomology group $H^{1}(\mbox{\P},\mathcal{O}(m))$ vanishes, the
anomaly formula of Bismut-Gillet-Soulé \cite{b-g-s} for the Quillen
metric on the determinant line $\bigwedge^{N_{m}}H^{0}(\mbox{\P},\mathcal{O}(m))$
reads as follows in our notation: \[
\log(\frac{\det\Delta_{u}}{\det\Delta_{0}})=\int\mbox{Td}(X,\omega_{0})\wedge\tilde{ch}(e^{-u}h_{0}^{\otimes m},h_{0}^{\otimes m})-\mathcal{L}_{m}(u),\]
 where $\mbox{Td}(X,\omega_{0})=(1+A\omega_{0})$ is the \emph{Todd
class }of $TX$ represented by the constant curvature metric $\omega_{0}$
expressed in terms of certain numerical constant $A,$ and $\tilde{ch}(e^{-u}h_{0}^{\otimes m},h_{0}^{\otimes m})=u+(u\omega_{u}+\omega_{0})/2$
is the \emph{Bott-Chern class} of the two metrics $e^{-u}h_{0}^{\otimes m}$
and $h_{0}^{\otimes m}$ on $\mathcal{O}(m)$ associated to the \emph{Chern
character} of $\mathcal{O}(m).$ In fact, $A=1,$ but the actual value
will turn out to be immaterial. Expanding gives \[
\log(\frac{\det\Delta_{u}}{\det\Delta_{0}})=\int udd^{c}u/2+B\int u\omega_{0}-\mathcal{L}_{m}(u)\]
for some constant $B.$ Since the left hand side is invariant under
translations of $u$ by constants it follows that $B=N.$ The previous
formula is precisely the one appearing in Prop 1 in \cite{g-s}),
since $h_{1}(\mathcal{O}(m))=0).$ Applying the inequality \ref{eq:ineq for l m in pf fang}
hence gives \[
\log(\frac{\det\Delta_{u}}{\det\Delta_{0}})\leq-\frac{1}{2}(1-\frac{m+1}{(m+2)})\int du\wedge d^{c}u=-\frac{1}{2}(\frac{1}{(m+2)})\int du\wedge d^{c}u\leq0\]
In particular, the lhs vanishes precisely when the gradient of $u$
does, i.e. when $u$ is a constant. This hence finishes the proof
of Corollary \ref{cor:max of det}. 
\begin{rem}
In the general anomaly formula in \cite{b-g-s} the metric $\omega_{0}$
is allowed to vary as well. In particular, when $L=\mathcal{O}(0)$
is the trivial holomorphic line bundle over $S^{2},$ the metric $h=1$
is kept constant, but the conformal metric $g_{u}=e^{-u}g_{0}$ on
$TS^{2}$ varies with $u,$ the anomaly formula in \cite{b-g-s} is
equivalent to Polyakov's formula and then $\log(\frac{\det\Delta_{g_{u}}}{\det\Delta_{g_{0}}})$
coincides with the functional $\mathcal{F}_{0}$ (up to a a multiplicative
constant) \cite{ch}. 
\end{rem}

\subsection{\label{sub:Arithemtic-applications}Arithmetic applications }

In this section we will briefly consider possible applications of
Theorem \ref{thm:main} to Arithmetic (Arakelov) geometry in the form
of\emph{ effective Riemann-Roch type inequalities}. In the general
setting $X$ will be the complex points of an \emph{arithmetic variety}
$X_{\Z}$ i.e. of a regular scheme, projective and flat over $\Z$
\cite{s--}. 

Consider for simplicity the case when $X=\P^{1}$ and denote as before
by $z$ the holomorphic variable in an affine piece of $\P^{1}.$
Let $h_{L^{2}}^{0}(\mathcal{O}(m),u)_{\Z}$ denote the logarithm of
the number of all polynomials $p_{m}$ in $z$ of degree at most $m$
such that $p_{m}$ has coefficients in $\Z+i\Z$ and such that $\left\Vert p_{m}\right\Vert _{u+m\psi_{0}}^{2}:=\int_{\C}\left|p_{m}(z)\right|e^{-(u+m\psi_{0})}\omega_{0}\leq1.$
The invariant $h_{L^{2}}^{0}(\mathcal{O}(m),u)$ is a non-standard
variant of basic invariants studied in Arakelov geometry, where one
usually only considers sections defined over $\Z$ (i.e. defined by
counting those polynomials $p_{m}$ as above which are invariant under
complex conjugation) and\emph{ sup-norms} instead of $L^{2}-$norms
(see sec. VIII, 2 in \cite{s--}). The arguments below can be adapted
to such invariants, but there will be an extra non-explicit term coming
from the distortion between the sup-norms and the $L^{2}-$norms determined
by $u.$ 

Fixing the base of monomials $(z^{j})$ in $H^{0}(\P^{1},\mathcal{O}(m))$
we may identify $H^{0}(\P^{1},\mathcal{O}(m))$ with $\C^{N_{m}}=\R^{2N_{m}},$
where $N_{m}=m+1.$ Then $h_{L^{2}}^{0}(\mathcal{O}(m),u)$ is simply
the logarithm of the number of points in the standard lattice in $\R^{2N_{m}}$
contained in a convex body determined by $u.$ Given Corollary \ref{cor:moser},
Minkowski's classical theorem is used to give an effective lower bound
on $h_{L^{2}}^{0}(\mathcal{O}(m),u):$ \[
h_{L^{2}}^{0}(\mathcal{O}(m),u)\geq(m+1)\mathcal{E}_{(m+2)\omega_{0}}(u)+C_{m},\]
 where $C_{m}$ is a certain explicit constant only depending on $m$
(see below). Since the argument is standard in Arakelov geometry (compare
p.164 in \cite{s--}) we will only briefly indicate it. First, by
basic linear algebra, we have that $\mathcal{L}_{m}(u)=\log(\mbox{Vol}\mathcal{B}(u+m\psi_{0})/\mbox{Vol}\mathcal{B}(m\psi_{0})),$
in terms of the volume of the unit-balls of the $L^{2}-$norms induced
by the weights $u+m\psi_{0}$ and $m\psi_{0},$ respectevely, wrt
Lesbegue measure in $\R^{2N_{m}}.$ Denote by $V_{m}$ the volume
of the unit-ball in $\R^{2N_{m}}.$ Then, using simple cocycle properties
of $\mathcal{L}_{m}(u)$ (defined in formula \ref{eq:def of l m in pf fang})
we get \[
\mathcal{L}_{m}(u)=\log(\mbox{Vol}\mathcal{B}(u+m\psi_{0})-\log V_{m}+Z_{m},\]
 where $Z_{m}=\log\mbox{det}_{0\leq i,j\leq m}(\int_{\C}\frac{z^{i}\bar{z}^{j}}{(1+z\bar{z})^{m}}\omega_{0}).$
Moreover, by Minkowski's theorem (see \cite{s--}) \[
h_{L^{2}}^{0}(\mathcal{O}(m),u)\geq\log(\mbox{Vol}\mathcal{B}(u+m\psi_{0})-(\log2)(2N_{m}).\]
 All in all this means, using Corollary \ref{cor:moser}, that \[
h_{L^{2}}^{0}(\mathcal{O}(m),u)\geq(m+1)\mathcal{E}_{(m+2)\omega_{0}}(u)+\log V_{m}-Z_{m}-(\log2)(2N_{m}).\]
 It can be checked that, when $u+m\psi_{0}=m\psi,$ where $dd^{c}\psi(z)\geq0$
the inequality above is an asymptotic equality (this is a special
case of formula \ref{eq:leading as of l}). Moreover, the rhs above
is equal to $m^{2}\mathcal{E}_{\omega_{+}}(\psi-\log^{+}\left|z\right|^{2})+o(m^{2}),$
where $\omega_{+}(z)=dd^{c}\log^{+}\left|z\right|^{2}.$ This means
that the lower bound above is consistent as it must with the asymptotic
arithmetic Riemann-Roch formula in \cite{g-sa} (see also Theorem
2' on p. 163 in \cite{s--}). In fact, the leading coefficent can
be shown to coincide in this case with a (normalized) arithmetic top-intersection
number, since $\mathcal{E}_{\omega_{+}}(\psi-\log^{+}\left|z\right|^{2})$
is precisely the classical weighted logarithmic density of $(\C,\psi)$
(see \cite{b-b} and references therein). The details are omitted.

\subsection{\label{sub:Alternative-proof-of}Alternative proof of uniqueness}

In this section we will show how to prove the {}``uniqueness'' in
Theorem \ref{thm:main} only using the regularity of the geodesics
furnished by Theorem \ref{thm:berman-dem} and the theory of fully
non-linear elliptic operators in $n$ complex dimensions (applied
to the Monge-Ampère operator on $X$ as in \cite{bl0}). In particular,
this latter theory amounts to the basic \emph{linear} elliptic estimates
for the Laplacian when $n=1.$

Recall that $W^{r,p}(X)$ denotes the Sobolev space of all distributions
$f$ on $X$ such that $f$ and the local derivatives of total order
$r$ are in $W^{0,p}(X):=L^{p}(X)$ (equvalently, all local derivatives
of total order $\leq r$ are in $L^{p}(X)).$ If $f$ is function
on $M=[0,1]\times X$ we will write \emph{$f_{t}\in W^{r,p}(X)$ uniformly
wrt $t$} if the corresponding Sobolev norms on $(X,\omega_{0})$
of $f_{t}$ are uniformly bounded in $t.$ We will also use the following
basic facts repeatedly:
\begin{itemize}
\item If $f$ is a function on $M$ such that \emph{$f_{t}\in W^{r,p}(X)$
uniformly wrt $t,$} then the distribution $f$ is in $W^{r,p}(\dot{M})$
and the corresponding Sobolev norms on $\dot{M}$ are bounded.
\item Partial derivatives of distributions commute
\item If $f,g\in W^{1,p}(X)$ for any $p>1.$ Then $fg\in W^{1,p}(X)$ for
any $p>1$ and Leibniz product rule holds for the distributional derivatives.
\end{itemize}
Note that as in section \ref{sub:Proof-of-Theorem-main} it will be
enough to prove that the geodesic $u_{t}$ is smooth wrt $(t,x)$
in the \emph{interiour} of $M.$ However, the arguments below will
even give uniform estimates on the local Sobolev norms up to the boundary
of $M.$

Assume now that the boundy data $u_{0}$ and $u_{1},$ defining the
geodesic $u_{t}$ are in $\mathcal{C}^{1,1}(X).$ Since $u_{t}$ is
convex in $t$ the right derivative (or tangent vector) $v_{t}(x):=\frac{d}{dt}_{+}u_{t}$
exists for all $(t,x).$
\begin{lem}
\emph{\label{lem:v is bded}The right tangent vector} \emph{$v_{t}$
of $u_{t}$ at $t$ is uniformly bounded on $M.$} \end{lem}
\begin{proof}
First observe that by the convexity in $t$\[
u_{t}-u_{0}\leq t(u_{1}-u_{0})\leq C_{1}t,\]
 using that $u_{0}$ and $u_{1}$ are continuos and hence uniformly
bouned on $X$ in the last step. Hence, $v_{t}\leq C.$ To get a lower
bound first observe that there is a {}``psh extension'' $\tilde{u}_{t}$
which is uniformly Lipshitz. Indeed, just take $\tilde{u}_{t}:=(1-t)u_{0}+tu_{1}+Ae^{t}$
for $A>>1.$ Using that $0\leq dd^{c}u_{0},\, dd^{c}u_{1}\leq C\omega_{0}$
it is straight-forward to check that $dd^{c}\tilde{U}+\pi^{*}\omega_{0}\geq0$
on $M$ for $A$ sufficently large. Since $U$ is defined by the upper
envelope \ref{eq:envelop} it follows that $\tilde{u}_{t}\leq u_{t}$
and hence \[
u_{t}-u_{0}\geq\tilde{u}_{t}-\tilde{u}_{0}\geq C_{2}t.\]
giving $v_{0}\geq C_{2}.$ Finally, by convexity we get $C_{2}\leq v_{0}\leq v_{t}\leq C_{1}$
which proves the lemma.\end{proof}
\begin{prop}
\label{pro:unique alt}Let $u_{0}$ be a critical point of $\mathcal{F}_{\omega_{0}}$
on $\overline{\mathcal{H}}_{\omega_{0}}\cap\mathcal{C}^{1,1}(X),$
$u_{1}$ an arbitrary element in $\overline{\mathcal{H}}_{\omega_{0}}\cap\mathcal{C}^{1,1}(X)$
and $u_{t}$ the geodesic connecting $u_{0}$ and $u_{1}.$ If $\mathcal{L}_{\omega_{0}}(u_{t})$
is affine, then there is an automorphism $S_{1}$ of $(X,L)$, homotopic
to the identity, such that $u_{1}-u_{0}=S_{1}^{*}\psi_{0}-\psi_{0}.$\end{prop}
\begin{proof}
\emph{Step 1: $u_{t}\in\mathcal{C}^{\infty}(X).$} First note that
by Theorem \ref{thm:berman-dem} $u_{t}\in\mathcal{C}_{\C}^{1,1}(X).$
Moreover, as shown in the beginning of section \ref{sub:Proof-of-Theorem-main}
it follows under the assumptions above that, for any $t,$ the function
$u_{t}$ satisfies the Euler-Lagrange equations \ref{eq:e-l for f}
on $X.$ Hence, just as in section \ref{sub:Proof-of-Theorem-main}
Blocki's complex version of the regularity result of Trudinger, now
applied to local patches of $\{t\}\times X$ immediately gives that
$u_{t}\in\mathcal{C}^{\infty}(X)$ (when $n=1$ this follows from
basic linear elliptic theory).

\emph{Step 2: $\Delta_{X}v_{t}\in L^{\infty}(X)$ uniformly wrt $t.$}
Differentiating the Euler-Lagrange equation wrt $t$ from the right
gives \begin{equation}
ndd^{c}v_{t}\wedge(\omega_{t})^{n-1}=\frac{d\beta_{u_{t}}(x)}{dt}_{+}=:R[v_{t}],\label{eq:linearized euler-l}\end{equation}
in the sense of currents. Of course, this would follow immediately
from the chain rule if $u_{t}$ were smooth in $(t,x).$ In the present
case it is proved in lemma \ref{lem:deriv of e-l} below. Moreover,
lemma \ref{lem:deriv of beta} in the appendix implies the bound \begin{equation}
\left\Vert R[v_{t}]/(\omega_{0})^{n}\right\Vert _{L^{\infty}(X)}\leq C\left\Vert v_{t}\right\Vert _{L^{\infty}(X)}\label{eq:ineq for r}\end{equation}
To see this, just note that \[
R[v]\leq2\left\Vert v\right\Vert _{L^{\infty}(X)}\int_{X}\left|K(x,y)\right|^{2}e^{-(\psi(x)+\psi(y))}=2\left\Vert v\right\Vert _{L^{\infty}(X)}\beta_{u},\]
 using the well-known {}``reproducing property'' of the Bergman
kernel (formula \ref{eq:integr out} in the appendix). By formula
\ref{eq:bounds on beta} in the appendix this proves the inequality
\ref{eq:ineq for r}.

Now, since $\omega_{t}>\delta\omega_{0},$ formula \ref{eq:linearized euler-l}
gives that the distribution $\Delta_{\omega_{t}}v_{t},$ where $\Delta_{\omega_{t}}$
is the Laplacian on $X$ wrt the metric $\omega_{t}:=\omega_{u_{t}},$
is in $L^{\infty}(X)$ uniformly wrt $t$ and \[
\left\Vert \Delta_{\omega_{t}}v_{t}\right\Vert _{L^{\infty}(X)}\leq C\left\Vert v_{t}\right\Vert _{L^{\infty}(X)}\leq C',\]
 by lemma \ref{lem:v is bded}. 

\emph{Step 3: $\Delta_{M}u\in W^{1,p}(M)$ for any $p\geq1.$} First
observe that by step 1 \begin{equation}
\partial_{z}(\partial_{z_{i}}\partial_{\bar{z}_{j}}u)\in L^{\infty}(X),\label{eq:zzz deriv}\end{equation}
 uniformly wrt $t.$ Also note that \begin{equation}
\partial_{t}(\partial_{z_{i}}\partial_{\bar{z}_{j}}u)\in L^{\infty}(X),\label{eq:tzz}\end{equation}
uniformly wrt $t.$ Indeed, $\partial_{t}(\partial_{z_{i}}\partial_{\bar{z}_{j}}u)=\partial_{z_{i}}(\partial_{\bar{z_{j}}}\partial_{t}u)=(\partial_{z_{i}}\partial_{\bar{z}_{j}})v_{t}\in L^{p}(X),$
uniformly wrt $t$\textbf{,} for any $p>1,$ by step 2 and local elliptic
estimates for $\Delta_{X}.$ Next, we will use that the following
identity proved in lemma \ref{lem:normal-normal} below:\begin{equation}
\partial_{t}\partial_{\bar{t}}u=\left|V_{t}\right|_{\omega_{t}}^{2}=\left|\partial_{\bar{z}}v_{t}\right|_{\omega_{t}}^{2},\label{eq:ma eq in prop unique}\end{equation}
 where $\left|V_{t}\right|_{\omega_{t}}^{2}$ denotes the point-wise
norm of $V_{t}$ wrt the metric $\omega_{t}$ (where we have used
that $\omega_{t}>0)$. First we have \begin{equation}
\partial_{z}(\partial_{t}\partial_{\bar{t}}u)=\partial_{z}\left|\partial_{\bar{z}}v_{t}\right|_{\omega_{t}}^{2}\in L^{p}(X),\label{eq:ztt}\end{equation}
 uniformly wrt $t,$ for any $p>1$ using Step 1 and Step 2 combined
with local elliptic estimates on $X$ for $\Delta_{X}.$ Next, \begin{equation}
\partial_{t}(\partial_{t}\partial_{\bar{t}}u)\in L^{p}(X),\label{eq:trippel deriv wrt t}\end{equation}
 uniformly wrt $t.$ Indeed, $\partial_{t}(\partial_{t}\partial_{\bar{t}}u)=\partial_{t}\left|\partial_{\bar{z}}\partial_{t}u\right|_{\omega_{t}}^{2}$
and since locally $\partial_{t}\omega_{t}=\partial_{t}(\partial_{z}\partial_{\bar{z}}u)$
\ref{eq:trippel deriv wrt t} follows from \ref{eq:tzz} and \ref{eq:ztt}
combined with Leibniz product rule. All in all this proves Step 3. 

Now by Step 3 and elliptic estimates for the Laplacian we have $u\in W^{3,p}(M).$
In particular, $u$ is locally in $\mathcal{C}^{2}(M).$ As a consequence
the proof of Theorem 2.6 in \cite{bern2} immediately gives that $V_{t}$
is a holomorphic vector field on $X$ for any $t.$ Finally, we will
recall a slight variant of the argument in \cite{bern2} which shows
that $\partial_{\bar{t}}V_{t}=0$ for $V_{t}$ seen as a distribution
on the interiour of $M.$ To simplify the notation we assume that
$n=1,$ but modulo the change to matrix notation the case $n>1$ is
the same. First we write \ref{eq:int multip} in the form \begin{equation}
\omega V=\partial_{\bar{z}}\partial_{t}u,\label{eq:int mult local}\end{equation}
where we have identified $V$ and $\omega$ with elements in $L^{p}(M)$
for $p>>1.$ By Leibniz rule \[
\partial_{\bar{t}}(V\omega)=(\partial_{\bar{t}}V)\omega+V(\partial_{\bar{t}}\omega)\]
Next, observe that \[
\partial_{\bar{t}}\omega=\partial_{\bar{t}}(\partial_{z}\partial_{\bar{z}}u)=\partial_{\bar{z}}(\partial_{\bar{t}}\partial_{z}u)=\partial_{\bar{z}}(\omega\bar{V}),\]
 using \ref{eq:int mult local} in the last step. Hence, since, as
shown above, $\partial_{\bar{z}}V=0,$ the two previous equations
together give \[
\partial_{\bar{t}}(V\omega)=(\partial_{\bar{t}}V)\omega+\partial_{\bar{z}}(V\omega\bar{V})=\partial_{\bar{z}}(\partial_{\bar{t}}\partial_{t}u),\]
 also using \ref{eq:int mult local} in the last step and commuting
$\partial_{\bar{z}}$ and $\partial_{\bar{t}}.$ Since, $V\omega\bar{V}=\left|V\right|_{\omega}^{2}$
it follows by \ref{eq:ma eq in prop unique} that $(\partial_{\bar{t}}V)\omega=0.$
But since, $\omega>0$ and $(\partial_{\bar{t}}V)$ is in $L^{p}(M)$
for all $p>1$ this forces $(\partial_{\bar{t}}V)=0$ a.e. on $M.$
In particular, $(\partial_{\bar{t}}V)=0$ as a distribution on $M.$
Hence, it follows that the distribution $V_{t}$ is in the null-space
of the $\bar{\partial}-$operator on $M.$ By local elliptic theory
it follows that $V_{t}$ is smooth and hence holomorphic in the interiour
of $M.$ Finally, the automorphism $S_{1}$ is obtained precisely
as in the end of section \ref{sub:Proof-of-Theorem-main}.\end{proof}
\begin{lem}
\emph{\label{lem:deriv of e-l}Under the assumptions in the previous
proposition the following holds:} \[
\frac{d}{dt}_{+}\int_{X}(\omega_{t})^{n}f=\int_{X}nv_{t}\wedge(\omega_{t})^{n-1}\wedge dd^{c}f,\]
 where $f$ is a given smooth function on $X.$ \end{lem}
\begin{proof}
To simplify the notation we assume that $n=2$ and $t=0,$ but the
general argument is completely similar (compare \cite{bbgz}). Expanding
and using that (by convexity) $(u_{t}-u_{0})/t$ decreases point-wise
to $u_{0},$ shows that it is equivalent to prove \[
\int(u_{t}-u_{0})(dd^{c}(u_{t}-\phi_{0}))\wedge dd^{c}f=o(t).\]
By partial integration the l.h.s is equal to \[
-\int d(u_{t}-u_{0})\wedge d^{c}(u_{t}-u_{0})\wedge dd^{c}f\]
Taking absolute values and using that $d(u_{t}-u_{0})\wedge d^{c}(u_{t}-u_{0})\geq0$
point-wise shows in turn that it is enough to prove the following\begin{equation}
\mbox{Claim:\ensuremath{\,\,}}\int d(u_{t}-u_{0})\wedge d^{c}(u_{t}-u_{0})\wedge\omega_{0}=o(t)\label{eq:claim in pf of diff lemma}\end{equation}
To this end first observe that \begin{equation}
\frac{d}{dt}_{t=0+}\mathcal{E}(u_{t})=\int_{X}v_{0}(\omega_{0})^{n}/(n!V)\label{eq:deriv of e along bad u}\end{equation}
To see this, note that since we have already shown that $\mathcal{E}(u_{t}),$
\emph{$\mathcal{F}(u_{t})(=\mathcal{E}(u_{t})-\mathcal{L}(u_{t}))$
and $\mathcal{L}(u_{t})$} are all affine (and even constant) \begin{equation}
0=\frac{d}{dt}_{t=0+}\mathcal{E}(u_{t})=\frac{d}{dt}_{t=0+}\mathcal{L}(u_{t})=\int_{X}v_{0}\beta_{u_{t}},\label{eq:deriv of e vanish}\end{equation}
 using formula \ref{eq:deriv of l} in the last step (and dominated
convergence for the sequence $(u_{t}-u_{0})/t$ converging to $v_{0}).$
Since, $u_{t}$ satisfies the Euler-Lagrange equations \ref{eq:e-l for f}
this proves formula \ref{eq:deriv of e along bad u}.
\end{proof}
Next, we will use the following well-known general identity (see \cite{bbgz}
or page 58-59 in \cite{ti}): \[
\mathcal{E}(u_{t})-\mathcal{E}(u_{0})-\int(u_{t}-u_{0})(\omega_{0})^{n}/n!=-J_{\omega_{0}}(\phi_{t}),\]
 in terms of the non-negative functional \[
J_{\omega_{0}}(u_{t})=c_{1}\int d(u_{t}-u_{0})\wedge d^{c}(u_{t}-u_{0})\wedge\omega_{0}+c_{2}\int d(u_{t}-u_{0})\wedge d^{c}(u_{t}-u_{0})\wedge\omega_{t}\]
where $c_{i}>0$ (compare formula \ref{eq:def of j}). But by \ref{eq:deriv of e vanish}
and the identity above\[
\frac{d}{dt}_{t=0+}J_{\omega_{0}}(\phi_{t})=0,\]
which by positivity implies the claim in formula \ref{eq:claim in pf of diff lemma}
and hence finishes the proof of the lemma.

In the previous proof we also used the following
\begin{lem}
\emph{\label{lem:normal-normal}Under the assumptions in the previous
proposition the following holds: $\partial_{t}\partial_{\bar{t}}u\in L^{\infty}(X)$}
\textup{uniformly in $t$} \emph{and \[
\partial_{t}\partial_{\bar{t}}u=\left|\bar{\partial}_{X}\partial_{t}u\right|_{\omega_{u_{t}}}^{2}.\]
}\end{lem}
\begin{proof}
By assumption the Monge-Ampère measure $(dd^{c}U+\pi_{X}^{*}\omega_{0})^{n+1}$
vanishes on $M.$ Moreover, by Step 1 in the proposition above $\Delta_{X}u_{t}\in C^{\infty}(X)$
for any $t$ with bounds on the Sobolev norms which are uniform wrt
$t.$ Combining this latter fact with lemma \ref{lem:v is bded} gives
that $U$ is Lipschitz on $M.$ Finally, as shown in Step 2 in the
proof of proposition above $\Delta_{X}\partial_{t}u_{t}\in L^{\infty}(X)$
uniformly wrt $t.$ We will next show that these properties are enough
to prove the lemma. As the statement is local we may as well consider
the restriction of $u:=U$ to an open set biholomorphic to a domain
in $\C^{n+1}=\C_{t}\times\C_{z}^{n}.$ Denote by $u^{\epsilon}$ the
local smooth function obtain as the convolution of $u$ with a fixed
local compactly supported smooth family of approximations of the identity.
Expanding gives \begin{equation}
(dd^{c}U+\pi_{X}^{*}\omega_{0})^{n+1}=(\partial_{t}\partial_{\bar{t}}u^{\epsilon}-\left|\partial_{\bar{z}}\partial_{t}u^{\epsilon}\right|_{\omega_{u^{\epsilon}}}^{2})(\omega_{u^{\epsilon}})^{n}\wedge dt\wedge d\bar{t}.\label{eq:expanding monge as c}\end{equation}
 Now since, by assumption, $\left|\partial_{\bar{z}}\partial_{t}u^{\epsilon}\right|_{\omega_{u^{\epsilon}}}^{2}\leq C$
the second term tends to $\left|\partial_{\bar{z}}\partial_{t}u\right|_{\omega_{u}}^{2})(\omega_{u})^{n}\wedge dt\wedge d\bar{t}$
weakly when $\epsilon\rightarrow0.$ Moreover, by assumption $u^{\epsilon}\rightarrow u$
uniformly locally and since the Monge-Ampère operator is continous,
as a measure, under uniform limits of psh functions \cite{de3} it
will now be enough to prove that \begin{equation}
(\partial_{t}\partial_{\bar{t}}u^{\epsilon})(\omega_{u^{\epsilon}})^{n}\wedge dt\wedge d\bar{t}\rightarrow(\partial_{t}\partial_{\bar{t}}u)(\omega_{u})^{n}\wedge dt\wedge d\bar{t}\label{eq:weak conv in pf lemma n-n}\end{equation}
 weakly, where the right hand sice is well-defined since $\partial_{t}\partial_{\bar{t}}u_{t}$
defines a positive measure on $\C^{n+1}$ and $(\omega_{u_{t}})^{n}/\omega_{0}^{n}$
is continous on $\C^{n+1}.$ To this end fix a test funtion $f$ i.e.
a smooth and compactly supported function on $\C^{n+1}.$ Then, with
$\int$ denoting the integral over $\C^{n+1},$ \[
\int f(\omega_{u^{\epsilon}})^{n}(\partial_{t}\partial_{\bar{t}}u^{\epsilon})\wedge dt\wedge d\bar{t}=:\int g_{\epsilon}(\partial_{t}\partial_{\bar{t}}u^{\epsilon})=-\int(\partial_{t}g_{\epsilon})(\partial_{\bar{t}}u^{\epsilon})\]
By assumption $(\partial_{t}g_{\epsilon})$ and $(\partial_{\bar{t}}u^{\epsilon})$
tend to $(\partial_{\bar{t}}u)$ and $(\partial_{\bar{t}}u),$ respectively
in $L^{p}(X)$ for any $p>1,$ uniformly wrt $t$ (more precisily
by the assumption on $\Delta_{X}u_{t}$ and the fact that $u$ is
Lipschitz). Hence, by Hölders's inequality \[
\int g_{\epsilon}(\partial_{t}\partial_{\bar{t}}u^{\epsilon})\rightarrow-\int(\partial_{t}g)(\partial_{\bar{t}}u).\]
 Finally, since $(\partial_{t}g)\in L^{\infty}(X)$ uniformly wrt
$t$ (by the assumption on $\Delta_{X}u_{t}$) and since $\partial_{t}\partial_{\bar{t}}u$
defines a positive measure, Leibniz rule combined with the dominated
convergence theorem gives (by a simple argument using a regularization
of $g)$ \[
-\int(\partial_{t}g)(\partial_{\bar{t}}u)=\int g(\partial_{t}\partial_{\bar{t}}u)\]
This proves \ref{eq:weak conv in pf lemma n-n} and hence finishes
the proof of the lemma. 
\end{proof}

\section{Applicatio\label{sec:Application-to-determinantal}n to $SU(2)-$invariant
determinantal point processes}

\emph{A random point process with $N$ particles} on a space $X$
wrt to a \emph{background measure} $\mu$ on $X,$ may be definied
as an ensemble of the form $(X^{N},\gamma_{N}),$ where \[
\gamma_{N}=\rho_{N}(x_{1},...,x_{N})d\mu^{\otimes N}\]
 and where the density $\rho_{N}(x_{1},...,x_{N})$ of the probability
measure $\gamma_{N}$ is assumed invariant under the action of the
symmetric group $S_{N},$ i.e. under permutations of the $x_{i}:$s.
The $N$-fold product $X^{N}$ is called the\emph{ $N-$particle configuration
space. }The random point process $(X^{N},\gamma_{N})$ determines
the random measure \begin{equation}
(x_{1},...,x_{N})\mapsto\sum_{i=1}^{N}\delta_{x_{i}},\label{eq:intro random measure}\end{equation}
 (i.e. a measure valued random variable) often called the \emph{empirical}
\emph{measure. }Given a, say continuous, function $u$ on $X$ one
defines the corresponding\emph{ linear statistic }as the random variable
obtained by contraction with the empirical measure: \begin{equation}
(x_{1},...,x_{N})\mapsto\sum_{i=1}^{N}u(x_{i})\label{eq:linear stat}\end{equation}
Using standard probability notation we will write $\E_{N}(Y):=\int Yd\gamma_{N}$
for the \emph{expectation} of a random variable $Y$ on $X^{N}$ and
its \emph{fluctuation} $\tilde{Y}$ is then the centered random variable
$\tilde{Y}-\E_{N}(Y).$ We also write $\mbox{Prob}{}_{N}(A):=\int_{A}\gamma_{N}.$

A special class of random point processes are given by the \emph{determinantal}
ones \cite{h-k-p}, which exhibit repulsion. . These have been mainly
studied when the background measure $\mu$ supported on $\C;$ notably
in the context of random matrix theory (cf. \cite{dia}). A general
complex geometric framework for determinatal random point processes
was introduced in \cite{be2}. Given a line bundle $L\rightarrow X$
over a compact complex manifold $X,$ a background measure $\mu$
and a weight $\psi_{0}$ of a metric on $L,$ the corresponding point
process is obtained by setting \[
\rho_{N}(x_{1},...,x_{N}):=\left|\det_{1\leq i,j\leq N}(s_{i}(x_{i}))_{i,j}\right|^{2}e^{-\psi_{0}(x_{1})}...e^{-\psi_{0}(x_{N})}/Z\]
in terms of any base $S=(s_{i})$ for the Hilbert space $H^{0}(X,L)$
equipped with the scalar product induced by $(\psi_{0},\mu).$ The
number $Z$ (called the \emph{partition function}) is the normalizing
constant ensuring that $\gamma_{N}$ is a probability measure on $X^{N}.$
Even if $Z$ does depend on the base $S$ the density $\rho_{N}$
does not. In the {}``adjoint'' setting considered in the present
paper where $L$ is replaced by $L+K_{X}$, there is no need to specify
the background measure $\mu_{0}$ (equivalenly, $\mu_{0}$ is taken
as any smooth volume form on $X$ which induces an Hermitian metric
on $K_{X}$ in such a way that the density $\rho_{N}$ is independent
of $\mu_{0}).$ 

The bridge between the point above processes and the subject of the
present paper is furnished by a formula which is a simple variant
of a well-known formula of Heine and Szegö in the theory of orthogonal
polynomials: \begin{equation}
\E_{N}(e^{-(u(x_{1})+...+u(x_{n})})=\det_{1\leq i,j\leq N}\left\langle s_{i},s_{j}\right\rangle _{(\psi_{0}+u,\mu)}\mbox{\emph{ }},\label{eq:expect is det}\end{equation}
i.e.$-\log\E_{N}(e^{-(u(x_{1})+...+u(x_{n})})=\mathcal{L}_{\omega_{0}}(u)$
in the notation of the previous sections  \cite{be2}. In fact, the
formula above is a simple consequence of the following identity \[
\int_{X^{N}}\left|\det_{1\leq i,j\leq N}(s_{i}(x_{i}))_{i,j}\right|^{2}e^{-\psi_{0}(x_{1})}...e^{-\psi_{0}(x_{N})}\mu=N!,\]
for $(s_{i})$ an orthonormal wrt the Hermitian products induced by
$(\psi_{0},\mu).$ Note that in the probability litterature $\E(e^{tY})$
is called the\emph{ moment generating function }of a given random
variable $Y.$\emph{ }

\subsection{The case of the two-sphere}

Let now $X$ be the two-sphere $S^{2}$ embedded as the unit-sphere
in Euclidian $\R^{3}$ and set \begin{equation}
\rho_{N}(x_{1},...,x_{N}):=\Pi_{1\leq i<j\leq N}\left\Vert x_{i}-x_{j}\right\Vert ^{2}/Z_{N}\label{eq:density on sphere}\end{equation}
written in terms of the ambient Euclidian norm in $\R^{3},$ where
$Z_{N}$ is the normalizing constant ensuring that $\gamma_{N}$ is
a probability measure on $X^{N}$ {[}in fact $1/Z_{N}=N^{N}\binom{N-1}{0}...\binom{N-1}{N-1}/N!]$
The background measure is taken as the induced volume (or rather area)
form $\omega_{0}$ on $S^{2}$ normalized to give unit volume to $S^{2}.$
Note that formula \ref{eq:distance}below shows that $g(x,y):=\log\left\Vert x-y\right\Vert ^{2}$
is the \emph{Green function} for $(S^{2},\omega_{0})$ (compare section
2.1 in \cite{be2b} and section 4 in \cite{z-z})

This random point process has two crucial properties: $i)$ it is
invariant under the isometry group of $S^{2}$ and $ii)$ it is\emph{
determinantal.}

In the physics litterature the ensemble above appears as the Gibbs
ensemble of a \emph{Coulomb gas} of unit-charge particles (i.e one
component plasma) confined to the sphere \cite{ca}. An interesting\emph{
random matrix} realiziation of this process was found very recently
in \cite{kr} (compare remark \ref{rem:random matrix on p1} below). 

In this probabilistic frame work Corollary \ref{cor:moser} may now
be formulated as the following {}``multi-particle Moser-Trudinger
inequality'' on $S^{2}$ (which is sharper then the one conjectured
in section 5 in \cite{f}). 
\begin{thm}
The following upper bound on the moment generating function of the
fluctuation of a linear statistic in the point process \ref{eq:density on sphere}
with $N-$particles on $S^{2}$ holds\begin{equation}
\log\E_{N}(e^{t(\widetilde{u(x_{1})}+\widetilde{u(x_{2})}+...+\widetilde{u(x_{N})})})\leq\frac{N}{N+1}\frac{t^{2}}{2}\left\Vert du\right\Vert ^{2}\label{eq:ineq in multi particle}\end{equation}
for any $t\in\R$ with equality iff $\omega_{0}-tdd^{c}u$ is the
pull-back of $\omega_{0}$ under a conformal transformation of $S^{2}.$\end{thm}
\begin{proof}
First observe that, in terms of the standard complex coordinate $z$
on $S^{2}$ with the north pole removed we have the basic identity
\begin{equation}
\left\Vert x_{1}-x_{2}\right\Vert ^{2}=\left|z_{i}-z_{j}\right|^{2}e^{-\psi_{0}(z_{1})}e^{-\psi_{0}(z_{2})},\label{eq:distance}\end{equation}
 (this is obvious for $z_{i}$ and $z_{j}$ on the unit-circle in
$\C$ and hence it holds everywhere, since the action of the group
$SU(2)$ by Möbius transformations acts transitively and preserves
both sides above). Substituting the previous formula in the definition
of $\rho_{N}$ above shows, using the standard product formula for
the Vandermonde determinant $\Delta(z_{1},...,z_{N}),$ that \[
\rho_{N}(x_{1},...,x_{N}):=\left|\Delta(z_{1},...,z_{N})\right|^{2}e^{-\psi_{0}(z_{1})}...e^{-\psi_{0}(z_{N})},\]
 where $\Delta(z_{1},...,z_{N_{k}})=\mbox{det}(s_{i}(z_{j})),$ with
$s_{i}(z)$ equal to the monomial $z^{j}.$ By the general formula
\ref{eq:expect is det} it follows that 

\[
-\log\E_{N}(e^{-(u(x_{1})+...+u(x_{N}))})=\mathcal{L}_{N-1}(u)\mbox{\emph{ }}\]
where $\mathcal{L}_{m}(u)$ is the functional \ref{eq:def of l m in pf fang}
. A simple scaling hence gives \[
\log\E_{N}(e^{-t(\tilde{u(x_{i})}+...+\tilde{u(x_{N})}})=-\mathcal{L}_{N-1}(u)-\E_{N}(\sum_{i=1}^{N}u(x_{i})).\]
Moreover, by general properties of determinantal point processes there
exists a function $\rho_{1}$ (called the one-point correlation function
\cite{h-k-p}) on $X$ such that \[
\E_{N}(\sum_{i=1}^{N}u(x_{i}))=\int u\rho_{1}\omega_{0},\]
 (in the present setting $\rho_{1}\omega_{0}$ may be identified with
the Bergman measure $\beta_{0}$ corresponding to $u=0;$ compare
formula \ref{eq:beta as one-pt}). Since $\rho_{N}$ and hence $\rho_{1}$
is invariant under isometries of $S^{2}$ it follows that $\rho_{1}$
is identically constant. Setting $u=1$ above forces in turn this
constant to be equal to $N.$ All in all this means that \[
\log\E_{N}(e^{-t(\tilde{u(x_{1})}+...+\tilde{u(x_{N})}})=-\mathcal{L}_{N-1}(u)-N\int u\omega_{0}\]
 Hence, applying Corollary \ref{cor:moser} finishes the proof of
the theorem (by replacing $u$ by $-tu)$.
\end{proof}
In the formula above we used the notation $\left\Vert du\right\Vert ^{2}$
for the squared $L^{2}-$norm on $S^{2}$ of the gradient of $u,$
written in conformally invariant notation as $\left\Vert du\right\Vert ^{2}:=\int_{S^{2}}du\wedge d^{c}u$
as in previous sections. Since the moment generating function of a
random variable controls the tail of its distribution we obtain the
following effective large deviation bound:
\begin{cor}
In the setting of the previous theorem the following large deviation
bound holds: for any given positive number $\lambda$: \[
\mbox{Prob}{}_{N}\{\frac{1}{N}(u(x_{1})+...+u(x_{N}))>\lambda\}\leq e^{-\frac{N^{2}\lambda^{2}}{2\left\Vert du\right\Vert ^{2}}\frac{N+1}{N}}\]
if the linear statistic \ref{eq:linear stat} is centered, i.e. if
its expected value vanishes.\end{cor}
\begin{proof}
The proof of this consequence of the previous theorem is a standard
application of Markov's inequality: for any given $t>0$ we have \[
\mbox{Prob}{}\{Y>1\}=\mbox{Prob}{}\{e^{tY}>e^{t}\}\leq e^{-t}\E(e^{tY}),\]
 where in our case $Y=\frac{1}{N\lambda}(u(x_{1})+...+u(x_{N}).$
By the previous theorem the rhs above is bounded by $e^{-t+ct^{2}}$
for $c=\frac{N}{N+1}\frac{1}{2}\left\Vert d(\frac{1}{N\lambda}u)\right\Vert ^{2}.$
Setting $t=1/2c$ (i.e. optimizing over $t)$ finally proves the corollary.
\end{proof}
Note that effective bounds as above are usually called \emph{Chernoff
bounds }in the classical probabilistic setting where the role of the
linear statistic is played by a random variable $Y$ of the form $Y=\frac{1}{N}(Y_{1}+...+Y_{N}),$
where $Y_{i}$ are \emph{independent }random variables with identical
\emph{symmetric} distribution. 

The bound in the previous corollary should be compared with the general
\emph{non-effective} bound \begin{equation}
\mbox{Prob}{}_{N_{k}}\{\frac{1}{N}(u(x_{1})+...+u(x_{N}))>\lambda\}\leq Ce^{-N^{2}/C},\label{eq:large dev bd non eff}\end{equation}
 where $C$ is a non-explicit constant, implied by the \emph{large
deviation principle} proved in \cite{be2b} for determinantal point
process in the general line bundle setting (compare the beginning
of this section). Note also that the bound \ref{eq:large dev bd non eff}
is essentially contained in the analysis in \cite{z-z}, since $X=\P^{1}$
in this case.

In the large $N-$limit the inequality in the previous theorem is
also closely related to a\emph{ Central Limit Theorem (CLT)} for the
linear statistic \ref{eq:linear stat}. Indeed, when $N$ tends to
infinity it can be shown that the inequality \ref{eq:ineq in multi particle}
becomes an asymptotic \emph{equality, i.e.} \begin{equation}
\lim_{N\rightarrow\infty}\log\E_{N}(e^{-t(\tilde{u(x_{i})}+...+\tilde{u(x_{N})}})=\frac{t^{2}}{2}\int_{S^{2}}du\wedge du^{c}\label{eq:free en limit clt}\end{equation}
 for any $t\in\R.$ In turn, by basic probability theory, this latter
fact can be shown to be equivalent to the following CLT: \[
\widetilde{u(x_{1})}+\widetilde{u(x_{2})}+...+\widetilde{u(x_{N})}\rightarrow\mathcal{N}(0,\frac{1}{2}\left\Vert du\right\Vert ^{2}),\]
in distribution, when $N\rightarrow\infty,$ where $\mathcal{N}(0,\frac{1}{2}\left\Vert du\right\Vert ^{2})$
is the centered normal variable with variance $\frac{1}{2}\left\Vert du\right\Vert ^{2}.$
See \cite{r-v2} for combinatorial proofs of this CLT on the sphere
and \cite{be2} for general results in the line bundle setting, using
Bergman kernel asymptotics. It is also interesting to compare with
the case of unitary random matrices, where the the role of the asymptotics
\ref{eq:free en limit clt} is played by \emph{Szegö's strong limit
theorem} \cite{dia}. See also \cite{joh} for the case of \emph{Hermitian}
random matrices and \cite{a-h-m1} for\emph{ normal} random matrices.

Loosely, speaking the CLT theorem above may also be formulated as
the statement that the \emph{ the potential of the fluctuations of
the empirical measure \ref{eq:intro random measure} on $S^{2}$ converges
in distribution to the Gaussian free field on $S^{2}$ }(see the introduction
in \cite{r-v2} and references therein).
\begin{rem}
\label{rem:random matrix on p1}Consider the probability measure on
$gl(N,\C)$ obtained by declaring the complex entries of an $N\times N$
matrix to be i.i.d complex Gaussians. Let $\Phi_{N}$ be the map defined
by \[
\Phi_{N}:\,(G_{1},G_{2})\mapsto(z_{1},...,z_{N})/S_{N},\]
where the $z_{i}:$s are the $N$ zeroes in $\C$ (taking multiplicities
into account) of $\det(G_{1}-zG_{2}),$ i.e. the eigen values of the
matrix $G_{2}(G_{1})^{-1},$ when $G_{1}$ is invertible. A remarkable
result in \cite{kr} says that the push-forward under $\Phi_{N}$
of the product probability measure on $gl(N,\C)\times gl(N,\C)$ is
precisely the random point process on $S^{2}$ with $N$ particles
defined by the density \ref{eq:density on sphere} (under stereographic
projection). 
\end{rem}

\section{\label{sec:Convergence-towards-Mabuchi's}Convergence towards Mabuchi's
K-energy}

In this section we will briefly consider the asympotic situation when
the ample line bundle $L$ is replaced by a multiple $kL$ for a large
positive integer $k$. Building on \cite{bern2} Berndtsson we will
relate the large $k$ asymptotics of $\mathcal{F}_{k\omega_{0}}$
to \emph{Mabuchi's K-energy.} The work \cite{bern2} was in turn inspired
by the seminal work of Donaldson \cite{don1} where a functional closely
related to $\mathcal{F}_{k\omega_{0}}$ was introduced (see section
\ref{sub:Comparison-with-Donaldson's} below). It should be pointed
out that there will be no original results in this section. But we
will give a simple proof of Theorem \ref{thm:mabuchi} below which
only uses the $C^{0}-$regularity of the geodesic connecting two given
smooth points in $\mathcal{H}_{\omega_{0}}$, which hopefully is of
some interest. See \cite{ch} for a proof which uses the $\mathcal{C}_{\C}^{1.,1}-$regularity
(Theorem \cite{ch}) in the case when the first Chern class of $X$
is assumed non-positive ). 

Fixing $\omega_{0}\in c_{1}(L)$ we willl take $k\omega_{0}$ as the
reference Kähler metric in $c_{1}(kL).$ Throughout the section $u$
will denote an element in $\mathcal{H}_{\omega_{0}}.$ For simplicity
we assume that that the volume $\mbox{Vol}(\omega_{0})=1.$ We will
write \[
\mathcal{F}_{k}(u):=k\mathcal{F}_{k\omega_{0}}-\bar{s}\mathcal{E}_{\omega_{0}}\]
 where $\bar{s}$ is the topological invariant of $L$ defined as
the average of the \emph{scalar curvature} $s_{u}$ of the Kähler
metric $\omega_{u}$ for any $u$ in $\mathcal{H}_{\omega_{0}}.$
The scaling in the definition of $\mathcal{F}_{k}(u)$ is motivated
by the fullowing asymptotic expansion of the differential of the functional
$\mathcal{L}_{k\omega_{0}}:$ \begin{equation}
(d\mathcal{L}_{k\omega_{0}})_{ku}=\omega_{u}^{n}(1+\frac{1}{k}s_{u}+o(1))/n!\label{eq:asym of dl}\end{equation}
 where the term $o(1)$ denotes a function which tends to zero uniformly
on $X$ (for $u$ fixed). 

Using formula \ref{eq:deriv of l} the proof of the previous formula
is reduced to the well-known asymptotics of the Bergman measure on
$kL+E,$ where $E$ is a given line bundle on $E,$ due to Tian-Catlin-Zelditch.
The reason that $s_{u}$ appears in the second term is that $E=K_{X}$
(see \cite{bern2}). In particular, we obtain \begin{equation}
(d\mathcal{F}_{k})_{ku}:=-(s_{u}-\bar{s})\omega_{u}^{n}+o(1)\label{eq:asympt of df}\end{equation}
Following Mabuchi \cite{m2,ti} the \emph{K-energy} (also called the
\emph{Mabuchi functional}) is defined, up to an additive constant,
as the primitive $\mathcal{M}$ on $\mathcal{H}_{\omega_{0}}$ of
the exact one-form defined by the measure valued function $u\mapsto(s_{u}-s)\omega_{u}^{n}$
on $\mathcal{H}_{\omega_{0}}.$ Hence, $u$ is a critical point of
$\mathcal{M}$ on $\mathcal{H}_{\omega_{0}}$ iff the Kähler metric
$\omega_{u}$ has \emph{constant scalar curvature}. We will denote
by $\mathcal{M}_{\omega_{0}}$ the K-energy normalized so that $\mathcal{M}_{\omega_{0}}(0)=0.$
Integrating along line segments in $\mathcal{H}_{\omega_{0}}$ and
using \ref{eq:asympt of df} immediately gives the asympotics \begin{equation}
\mathcal{F}_{k}(u)=-\mathcal{M}_{\omega_{0}}(u)+o(1).\label{eq:asympt of f}\end{equation}
For the most general version of the following theorem see \cite{c-t}.
\begin{thm}
\label{thm:mabuchi}Assume that the Kähler metric $\omega_{u}$ has
constant scalar curvature. Then $u$ minimizes Mabuchi's K-energy
$\mathcal{M}_{\omega_{0}}$ on $\mathcal{H}_{\omega_{0}}.$\end{thm}
\begin{proof}
By the cocycle property of $\mathcal{M}_{\omega_{0}}$ we may as well
assume that $u=0$ in the statement above. Now fix an arbitrary $u$
in $\mathcal{H}_{\omega_{0}}$ and take the $C^{0}-$geodesics $u_{t}$
connecting $0$ and $u.$ Given a positive integer $k$ the fact that
$\mathcal{F}_{k}$ is concave along $u_{t}$ (compare the proof of
Theorem \ref{thm:main}) immediately gives \[
\mathcal{F}_{k}(u)\leq\mathcal{F}_{k}(0)+\frac{d}{dt}_{t=0+}\mathcal{F}(u_{t}).\]
Combining formulas \ref{eq:asympt of f}, \ref{eq:asympt of df} then
gives \[
\mathcal{F}_{k}(u)\leq-\mathcal{M}_{\omega_{0}}(u)+\int_{X}(s_{u}-\bar{s})\omega_{u}^{n}v_{0}+\int_{X}o(1)\omega_{u}^{n}v_{0},\]
 where $v_{0}=\frac{du}{dt}_{t=0+}$. But by lemma \ref{lem:v is bded}
we have that $v_{0}$ is uniformly bounded (in fact it is enough to
know that its $L^{1}-$norm is uniformly bounded, which can be proved
as in \cite{bbgz}. Letting $k$ tend to infinity the assumption on
$u$ hence gives \[
-\mathcal{M}_{\omega_{0}}(u)\leq-\mathcal{M}_{\omega_{0}}(0),\]
 which hence finishes the proof of the theorem.
\end{proof}
In particular, the proof above shows that, $\mathcal{M}_{\omega_{0}}$
is {}``convex along a geodesic'', in the sense that it is the point-wise
limit of the \emph{convex} functionals $\mathcal{F}_{k}$ along a
geodesic connecting two points in $\mathcal{H}_{\omega_{0}},$ only
using the $C^{0}-$regularity of the corresponding geodesic. Note
however that the definition of $\mathcal{M}_{\omega_{0}}$ as given
above does not even make sense unless $u_{t}$ is in $\mathcal{C}^{4}(X)$,
for $t$ fixed and $\omega_{t}>0$ (the smoothness assumption may
be relaxed to $u_{t}\in\mathcal{C}_{\C}^{1,1}(X)$ using the alternative
formula for $\mathcal{M}_{\omega_{0}}$ from \cite{ti2,ch2}). In
the case when the geodesic $u_{t}$ is assumed \emph{smooth} and $\omega_{t}>0$
the argument in the proof of the theorem above is essentially contained
in \cite{bern2}. In this latter case the convexity statement seems
to first have appeared in \cite{m} (see also \cite{d00} ). In \cite{don1}
the previous theorem was proved using the deep results in \cite{do1}
and the {}``finite dimensional geodesics'' in approximations of
$\mathcal{H}_{\omega_{0}}$ as briefly explained in the following
section.

\subsection{\label{sub:Comparison-with-Donaldson's}Comparison with Donaldson's
setting and balanced metrics}

In the setting of Donaldson \cite{don1} the role of the space $H^{0}(X,L+K_{X})$
is played by the space $H^{0}(X,L).$ Any given function $u$ in $\mathcal{H}_{\omega_{0}}$
induces an Hermitian norm $Hilb(u)$ on $H^{0}(X,L)$ defined by \[
Hilb(u)[s]^{2}:=\int_{X}\left|s\right|^{2}e^{-(\psi_{0}+u)}(\omega_{u})^{n}/n!\]
 Then the functional that we will refer to as $\mathcal{L}_{D}(u),$
which plays the role of $\mathcal{L}_{\omega_{0}}(u)$ in Donaldson's
setting, is defined as in formula \ref{eq:def of l intro}, but using
the scalar product on $H^{0}(X,L)$ corresponding to $Hilb(u).$ With
this definition it turns out that $\mathcal{L}_{D}(u)$ is \emph{concave}
along smooth geodesics (see Theorem 3.1 in \cite{bern2} for a generalization
of this fact). However, it does not appear to be concave along a general
psh paths, which makes approximation more difficult in this setting.
Moreover, Theorem 2 in \cite{don1} says that the critical points
of $\mathcal{E}-\mathcal{L}_{D}$ are in fact \emph{minimizers.} %
\footnote{Comparing with the notation in \cite{don1}, $\mathcal{L}_{D},$ $\mathcal{E}$
and $u$ correspond to $-\mathcal{L},$ $-I$ and $-\phi,$ respectively. %
}A major technical advantage of Donaldson's setting is that the critical
points (which are called \emph{balanced} in \cite{don1}) of the functional
$\mathcal{E}-\mathcal{L}_{D}$ acting on all of $\mathcal{C}^{\infty}(X)$
are automatically of the form \begin{equation}
\psi=\log({\textstyle \frac{1}{N}\sum_{i}}\left|S_{i}\right|^{2})\label{eq:bergm metric}\end{equation}
 fore some base $(S_{i})$ in $H^{0}(X,L).$ In particular, $u$ is
automatically in $\mathcal{H}_{\omega_{0}}$ (assuming that $L$ is
very ample). This is then used to replace the space $\mathcal{H}_{\omega_{0}}$
by the sequence of \emph{finite dimensional} symmetric spaces $GL(N,\C)/U(N)$
corresponding to the set of metrics on $L$ of the form \ref{eq:bergm metric}
(called Bergman metrics). In particular, the new geodesics, defined
wrt the Riemannian structure in the symmetric space $GL(N,\C)/U(N)$
are automatically smooth and the analysis in \cite{don1} is reduced
to this finite dimensional situation. 

Note also that in this setting there is a sign difference in the expansion
\ref{eq:asym of dl}, where $s_{u}$ is replaced by $-s_{u}.$ As
a consequence, in Donaldson's case the functional corresponding to
$\mathcal{F}_{k}$ converges to $\mathcal{M}_{\omega_{0}}$ (without
the minus sign!), which hence becomes convex along smooth geodesics,
which is consistent with the conclusion reached above, as it must. 

Finally, note that combining the upper bound in Theorem \ref{thm:main}
combined with the lower bound coming from a (slight variant) of Donaldson's
scalar product on $H^{0}(X,L+K_{X})$ (i.e. using Theorem 2 in \cite{don1})
gives \[
-C+k\mathcal{E}_{k\omega_{0}}(u)\leq\mathcal{L}_{k\omega_{0}}(ku)\leq k\mathcal{E}_{k\omega_{0}}(u),\]
 where $C$ is a positive constant proportional to $\left\Vert \omega_{u}^{n}/\omega_{0}^{n}\right\Vert _{L^{\infty}(X)}.$
In particular, this yields the asymptotics \begin{equation}
\mathcal{L}_{k\omega_{0}}(ku)=k\mathcal{E}_{\omega_{0}}(u)+O(1),\label{eq:leading as of l}\end{equation}
 which is a well-known result. In fact, it may be directly obtained
using the leading term in the asymptotics \ref{eq:asym of dl}( see
\cite{b-b} for the generalization to non-positively curved metrics).

\section{Appendix}

\subsection{\label{sub:Bergman-kernels}Bergman kernels}

Given a function $u$ corresponding to the weight $\psi:=\psi_{0}+u$
on the line bundle $L$ we denote by $K_{u}(x,y)$ the \emph{Bergman
kernel} of the Hilbert space $(H^{0}(X,L+K_{X}),\left\langle \cdot,\cdot\right\rangle _{\psi_{0}+u}),$
i.e. \[
K_{u}(x,y):=i^{n^{2}}\sum_{i=1}^{N}s_{i}(y)\wedge\bar{s_{i}(x),}\]
represented in terms of a given orthonormal base $(s_{i})$ in $(H^{0}(X,L+K_{X}),\left\langle \cdot_{i},\cdot\right\rangle _{\psi_{0}+u}).$
This kernel may be caracterized as the integral kernel of the correponding
orthogonal projection $\Pi_{u}$ onto $(H^{0}(X,L+K_{X}),\left\langle \cdot_{i},\cdot\right\rangle _{\psi_{0}+u}),$
i.e. for any smooth section $s$ of $L+K_{X}$

\begin{equation}
(\Pi_{u}s)(x)=\int_{X_{y}}s(y)\wedge\bar{K}(x,y)e^{-(\psi(y)}\label{eq:pi in terms of k}\end{equation}
The \emph{Toeplitz operator $T[f]$ with symbol $f\in C^{0}(X),$}
acting on $(H^{0}(X,L+K_{X}),\left\langle \cdot_{i},\cdot\right\rangle _{\psi_{0}})$
(defined below formula \ref{eq:l as toeplitz det intro}) may then
be expressed as \begin{equation}
(T[f])(x)=\int_{X_{y}}f(y)s(y)\wedge\bar{K}(x,y)e^{-\psi(y)}\label{eq:toeplitz as k}\end{equation}
Applying \ref{eq:pi in terms of k} $K_{u}(x,\cdot)$ gives the following
{}``integrating out'' formula \begin{equation}
N\beta_{u}(x):=K_{u}(x,x)e^{-\psi(x)}:=\int_{X_{y}}\left|K(x,y)\right|^{2}e^{-(\psi(x)+\psi(y)}\label{eq:integr out}\end{equation}
When studying the dependence of $\beta_{u}$ on $u$ it is useful
to express $\beta_{u}(x)$ as the normalized\emph{ one-point correlation
measure }of the determinantal point process induced by $\psi$ (see
section \ref{sec:Application-to-determinantal} and \cite{be2}) \begin{equation}
\beta_{u}(x)=\frac{1}{N}\E_{\psi}(\sum_{i=1}^{N}\delta_{x_{i}})=\int_{X^{N-1}}\left|(\det S_{0})(x,x_{2},...,x_{N}\right|^{2}e^{-\psi(x)}e^{-\psi(x_{2})}...e^{-\psi(x_{N})}/Z_{\psi}\label{eq:beta as one-pt}\end{equation}
 In particular, the map $(x,t)\mapsto(\beta_{u_{t}}(x)/\omega_{0}^{n})$
is \emph{continuous} if $u_{t}$ is a continuous path and hence there
is a positive constant $C$ such that\begin{equation}
1/C\leq(\beta_{u_{t}}(x)/\omega_{0}^{n})\leq C\label{eq:bounds on beta}\end{equation}
 on $[0,1]\times X,$ if $L+K_{X}$ is globally generated, i.e. if
$\beta_{u_{t}}(x)>0$ point-wise. Formula \ref{eq:beta as one-pt}
also shows, by the dominated convergence theorem, that $\frac{d\beta_{u_{t}}(x)}{dt}_{t=0+}$
exists under the assumptions in the following lemma. 
\begin{lem}
\label{lem:deriv of beta}Let $u_{t}$ be a family of continuous functions
on $X$ such that the right derivative $v_{t}:=\frac{du_{t}}{dt}_{+}$
exists and is uniformly bounded on $[0,1]\times X.$ Then \begin{equation}
R[v](x):=\frac{d\beta_{u_{t}}(x)}{dt}_{t=0+}=\int_{X_{y}}\left|K_{u}(x,y)\right|^{2}e^{-(\psi_{0}(x)+\psi_{0}(y)}v_{0}(y)-\beta_{u_{t}}(x)v_{0}(x)\label{eq:r operator}\end{equation}
.\end{lem}
\begin{proof}
The proof of the formula was obtained in \cite{be} (formula $5),$
at least in the smooth case. For completeness we recall the simple
proof. By the discussion above we may differentiate formula \ref{eq:integr out}
and use Leibniz product rule to get \[
\partial_{t}(K_{t}(x,x))=2\mbox{Re }\int_{X_{y}}\partial_{t}(K_{t}(x,y)\wedge\bar{K}(x,y)e^{-\psi_{t}(y)}-\int_{X_{y}}\left|K_{t}(x,y)\right|^{2}(\partial_{t}\psi_{t}(y))e^{-(\psi_{t}(x)+\psi_{t}(y)}\]
Applying formula \ref{eq:pi in terms of k} to the holomorphic section
$s(\cdot)=\partial_{t}K_{t}(x,\cdot)$ shows that the second term
above equals $2\partial_{t}(K_{t}(x,x)).$ Hence, \[
\partial_{t}(K_{t}(x,x))=\int_{X_{y}}\left|K_{t}(x,y)\right|^{2}(\partial_{t}u_{t}(y))e^{-(\psi_{t}(x)+\psi_{t}(y)},\]
 which proves the lemma, since $N\beta_{u}(x)=K(x,x)e^{-\psi(x)}$.
\end{proof}

\subsection{A {}``Bergman kernel proof'' of Theorem \ref{thm:(Berndtsson)-Let-}}

Let $\psi_{t}:=\psi_{0}+u_{t}.$ As will be shown below, differentiating
$\mathcal{L}_{\omega_{0}}(u_{t})$ gives \begin{equation}
\partial_{t}\partial_{\bar{t}}\mathcal{L}_{\omega_{0}}(u_{t})=\frac{1}{N}\sum_{i=1}^{N}(\left\Vert (\partial_{t}\partial_{\bar{t}}u_{t})s_{i}\right\Vert _{\psi_{t}}^{2}-\left\Vert (\partial_{\bar{t}}u_{t}s_{i})-\Pi_{u_{t}}(\partial_{\bar{t}}u_{t}s_{i})\right\Vert _{\psi_{t}}^{2}),\label{eq:second der of l as proj}\end{equation}
where $(s_{i})$ is orthnormal wrt $\psi_{t}=\psi_{0}+u_{t}.$ Given
this formula the argument proceeds exactly as in \cite{bern2}; by
the definition of $\Pi_{u_{t}},$ the second term inside the sum is
the $L^{2}$- norm of the solution $s$ to the inhomogenous $\bar{\partial}-$equation
on $X:$ \[
\bar{\partial_{X}}s=\bar{\partial_{X}}(\partial_{t}u))s_{i},\]
 which has minimal norm wrt $\left\Vert \cdot\right\Vert _{\psi_{t}}^{2}.$
Now the Hörmander-Kodaira $L^{2}-$inequality for the solution gives
\begin{equation}
i^{n^{2}}\int_{X}s\wedge\bar{s}e^{-\psi_{t}}\leq i^{n^{2}}\int_{X}\left|\bar{\partial_{X}}(\partial_{t}u_{t})\right|_{\omega_{u_{t}}}^{2}s\wedge\bar{s}e^{-\psi_{t}},\label{eq:horm est}\end{equation}
 using that $\omega_{u_{t}}>0.$ Hence, by formula \ref{eq:second der of l as proj},
\[
\partial_{t}\partial_{\bar{t}}\mathcal{L}_{\omega_{0}}(u_{t})\geq\frac{1}{N}\sum_{i=1}^{N}(\left\Vert (\partial_{t}\partial_{\bar{t}}u_{t})-\left|\bar{\partial_{X}}(\partial_{t}u_{t})\right|_{\omega_{u_{t}}}^{2})s_{i}\right\Vert _{\psi_{t}}^{2}\]
But since, by assumption, $(dd^{c}U+\pi_{X}^{*}\omega_{0})^{n+1}\geq0$
the rhs is non-negative (compare formula \ref{eq:expanding monge as c}),
which proves that $\mathcal{L}_{\omega_{0}}(u_{t})$ is \emph{convex
}wrt real $t.$ Note that $\mathcal{L}_{\omega_{0}}(u_{t})$ is \emph{affine}
precisely when \ref{eq:horm est} is an\emph{ equality. }By examining
the Bochner-Kodaira-Nakano-Hörmander\emph{ identity} implying the
inequality \ref{eq:horm est} one sees that the remaining term appearing
in the identity has to vanish. In turn, this is used to show that
the vector field $V_{t}$ defined by formula \ref{eq:int multip}
has to be \emph{holomorphic} on $X$ (see \cite{bern2}). Integrating
$V_{t}$ finally gives the existence of the automorphism $S_{1}$
in Theorem \ref{thm:(Berndtsson)-Let-}, as explained in section \ref{sub:Proof-of-Theorem-main}.

In \cite{bern2} formula \ref{eq:second der of l as proj} was derived
using the general formalism of holomorphic vector bundles and their
curvature. We will next give an alternative {}``Bergman kernel proof''.
First formula \ref{eq:r operator} and Leibniz product rule give \[
\partial_{t}\partial_{\bar{t}}\mathcal{L}_{\omega_{0}}(u_{t})=\int_{X}(\partial_{t}\partial_{\bar{t}}u_{t})\beta_{u_{t}}+\frac{d\beta_{u_{t}}}{dt}_{t=0+}(\partial_{t}u_{t})\]
Next, by formula \ref{eq:r operator} the second term may be expressed
in terms of the Bergman kernel $K_{t}(x,y)$ associated to the weight
$\psi_{t}$ as \[
\frac{1}{N}\int_{X\times X}\left|K_{t}(x,y)\right|^{2}e^{-(\psi_{t}(x)+\psi_{t}(y)}((\partial_{t}u_{t})(x)(\partial_{t}u)(y)-\int_{X}\beta(\partial_{t}u_{t})^{2},\]
 By simple and  well-known identities for Toeplitz operators this
last expression, for $t=0,$ is precisely the trace of the operator
$T[\partial_{t}u_{t}])^{2}-T[(\partial_{t}u_{t})^{2}].$ All in all
we obtain \[
\partial_{t}\partial_{\bar{t}}\mathcal{L}_{\omega_{0}}(u_{t})=\frac{1}{N}\mbox{Tr}(T[\partial_{t}\partial_{\bar{t}}u_{t}]+(T[\partial_{t}u_{t}])^{2}-T[(\partial_{t}u_{t})^{2}]),\]
for $t=0.$ Expanding in terms of an orthonormal base $s_{i}$ hence
gives \[
\partial_{t}\partial_{\bar{t}}\mathcal{L}_{\omega_{0}}(u_{t})=\frac{1}{N}\sum_{i=1}^{N}(\left\Vert (\partial_{t}\partial_{\bar{t}}u_{t})s_{i}\right\Vert _{\psi_{0}+u_{t}}^{2}+\left\Vert \Pi_{u_{t}}(\partial_{t}u_{t}s_{i})\right\Vert _{\psi_{0}+u_{t}}^{2}-\left\Vert \partial_{t}u_{t}s_{i}\right\Vert _{\psi_{0}+u_{t}}^{2}),\]
for $t=0$ (and hence for all $t$ by symmetry) which finally proves
\ref{eq:second der of l as proj}, using {}``Pythagora's theorem''.

\end{document}